\documentclass[10pt,reqno]{amsart}
\usepackage[a4paper,top=3cm, bottom=3cm,left=3cm, right=3cm,twoside]{geometry}

\usepackage{amsthm}
 \usepackage{amsmath}
 \usepackage{amsfonts}
 \usepackage{amssymb}
 \usepackage{amscd}
 \usepackage[T1]{fontenc}
 \usepackage[utf8x]{inputenc}
 \usepackage[english]{babel}
 \usepackage[mathscr]{eucal}
 \usepackage{lmodern}
 \usepackage{bera}
 \usepackage{float} 
 \usepackage{bold-extra}
 \usepackage{textcomp}
 \usepackage{graphicx}
 \usepackage{caption}
\usepackage{subcaption} 
 \usepackage[dvipsnames,table]{xcolor}
 \usepackage{mathtools} 
 \usepackage[shortlabels]{enumitem}
 \usepackage{tikz} 
 \usepackage{tikz-cd}
 \usepackage{wasysym}
 \usepackage{todonotes}
  \usepackage{fancybox}
 \usepackage{fancyhdr}
\usepackage{hyperref}

 \theoremstyle{definition}  
  \newtheorem{definition}{Definition}[section]
  \newtheorem{example}[definition]{Example}

   \newtheorem{remark}[definition]{Remark}
   
  \theoremstyle{plain}  
  \newtheorem{theorem}[definition]{Theorem}
  \newtheorem{lemma}[definition]{Lemma}
  \newtheorem{proposition}[definition]{Proposition}
  \newtheorem{corollary}[definition]{Corollary}

  \theoremstyle{remark}

  \newtheorem*{notation}{Notation}

 \renewcommand{\it}[1]{\textit{#1}}
 
 \renewcommand{\sf}[1]{\textsf{#1}}

 \newcommand{\mbb}[1]{\mathbb{#1}}
 \newcommand{\mcl}[1]{\mathcal{#1}}

 \newcommand{\mbf}[1]{\mathbf{#1}}
 \newcommand{\msc}[1]{\mathscr{#1}}
 
 \newcommand{\msf}[1]{\text{\small$\sf{#1}$}}
  
 \newcommand{\ol}[1]{\overline{#1}}
 \newcommand{\ul}[1]{\underline{#1}}
 \newcommand{\wtilde}[1]{\widetilde{#1}}

 \newcommand{\abs}[1]{\left\lvert#1\right\rvert}

 \newcommand{\norm}[1]{\left\lVert#1\right\rVert}
 \newcommand{\bnorm}[1]{\bigl\lVert#1\bigr\rVert}

 \newcommand{\B}[1]{\msc{B}({#1})} 
 \newcommand{\K}[1]{\mcl{K}({#1})}

 \newcommand{\ip}[1]{\langle#1\rangle}

 \newcommand{\ran}[1]{\sf{range}(#1)}
 \newcommand{\clran}[1]{\ol{\sf{range}}(#1)}
 \renewcommand{\ker}[1]{\sf{ker}(#1)}

 \newcommand{\mscriptsize}[1]{{\setlength{\arraycolsep}{.3ex}\text{\scriptsize$#1$}}}

 \newcommand{\Matrix}[1]{\begin{bmatrix}#1\end{bmatrix}}
 \newcommand{\sMatrix}[1]{\mscriptsize{\Matrix{#1}}}

 \DeclareMathOperator{\id}{\sf{id}}

 \DeclareMathOperator{\lspan}{\sf{span}}
 \DeclareMathOperator{\cspan}{\ol{\lspan}}

 \numberwithin{equation}{section}
 \allowdisplaybreaks[4] 
 \setlist[enumerate]{font=\upshape,noitemsep, topsep=0pt} 
 \setlist[itemize]{noitemsep, topsep=0pt}

\makeatletter
\@namedef{subjclassname@2020}{\textup{2020} Mathematics Subject Classification}
\makeatother

\begin{document}

\title[extreme points]{$C^*$-extreme points of unital completely positive maps on real $C^*$-algebras}

\author{Anand O. R}
\address{Indian Institute of Technology Madras, Department of Mathematics, Chennai, Tamilnadu 600036, India}
\email{anandpunartham@gmail.com, ma20d006@smail.iitm.ac.in}

\author{K. Sumesh}
\address{Indian Institute of Technology Madras, Department of Mathematics, Chennai, Tamilnadu 600036, India}
\email{sumeshkpl@gmail.com, sumeshkpl@iitm.ac.in}

\author{Arindam Sutradhar}
\address{Indian Institute of Technology Madras, Department of Mathematics, Chennai, Tamilnadu 600036, India}
\email{arindam1050@gmail.com, ic40339@imail.iitm.ac.in}

\date{\today}

\begin{abstract}
 In this paper, we investigate the general properties and structure of $C^*$-extreme points within the $C^*$-convex set $\mathrm{UCP}(\mcl{A},\B{\mcl{H}})$ of all unital completely positive (UCP) maps from a unital real $C^*$-algebra $\mcl{A}$ to the algebra $\B{\mcl{H}}$ of all bounded real linear maps on a real Hilbert space $\mcl{H}$. We analyze the differences in the structure of $C^*$-extreme points between the real and complex $C^*$-algebra cases.  In particular, we show that the necessary and sufficient conditions for a UCP map between matrix algebras to be a $C^*$-extreme point are identical in both the real and complex matrix algebra cases. We also observe significant differences in the structure of $C^*$-extreme points when $\mcl{A}$ is a commutative real $C^*$-algebra compared to when $\mcl{A}$ is a commutative complex $C^*$-algebra. We provide a complete classification of the $C^*$-extreme points of $\mathrm{UCP}(\mcl{A},\B{\mcl{H}})$, where $\mcl{A}$ is a unital commutative real $C^*$-algebra and $\mcl{H}$ is a finite-dimensional real Hilbert space. As an application, we classify all $C^*$-extreme points in the $C^*$-convex set of all contractive skew-symmetric real matrices in $M_n(\mbb{R})$.
\end{abstract}



\maketitle

\tableofcontents

\section{Introduction}

    The notion of $C^*$-convexity and $C^*$-extreme points was introduced by Loebl and Paulsen \cite{LoPa81} for subsets of a complex $C^*$-algebra. This concept was further explored in \cite{HMP81, FaMo93, Mag01}. In \cite{FaMo97}, Farenick and Morenz extended the notion of $C^*$-convexity and $C^*$-extreme points to the convex set $\mathrm{UCP}(\mcl{A},\B{\mcl{H}})$, which consists of all unital completely positive (UCP) maps from a complex $C^*$-algebra $\mcl{A}$ to the algebra $B(\mcl{H})$ of all bounded linear maps on a complex Hilbert space $\mcl{H}$. They provided a complete description of the $C^*$-extreme points of $\mathrm{UCP}(\mcl{A},\B{\mcl{H}})$ when $\mcl{A}$ is either a commutative $C^*$-algebra or the algebra $M_n(\mbb{C})$ of all $n\times n$ complex matrices, and $\mcl{H}$ is a finite-dimensional Hilbert space. Subsequently, in \cite{FaZh98} Farenick and Zhou gave a comprehensive description of the $C^*$-extreme points when $\mcl{A}$ is an arbitrary unital $C^*$-algebra and $\mcl{H}$ is a finite-dimensional Hilbert space. For an abstract characterization of the $C^*$-extreme points of $\mathrm{UCP}(\mcl{A},\B{\mcl{H}})$ in the general case, we refer to \cite{FaZh98, Zho98}. When $\mcl{A}$ is a commutative unital $C^*$-algebra, the $C^*$-extreme points of $\mathrm{UCP}(\mcl{A},\B{\mcl{H}})$ were studied in \cite{Gre08,Gre09}. More recently, $C^*$-extreme points have been extensively studied in various contexts (see  \cite{BBK21,BhKu22,BDMS23, BaHo24, AnSu25}). 

    The study of $C^*$-extreme points within $C^*$-convex subsets of completely positive (CP) maps has primarily concentrated on complex $C^*$-algebras. This article begins an exploration into the $C^*$-extreme points of UCP maps between real $C^*$-algebras. While one might speculate that a similar theory could be developed for the real case via the process of "complexification", technical challenges arise when attempting to translate results from the complex setting to the real context through this method. For an overview of these challenges, see the recent survey article \cite{MMPS22}. Although operator algebraists historically overlooked the theory of real $C^*$-algebras, this trend has shifted in recent years. Independently, the theory of real $C^*$-algebras has gained notable applications \cite{Ros16, CDPR23}. Additionally, in \cite{Ruan03, BlTe21, CDPR23, BlCeKa24}, the study of completely positive maps between real $C^*$-algebras was undertaken, which further motivates the investigation of $C^*$-extreme points in the real case.

    This paper is organized as follows. In Section 2, we recall the basic definitions and results needed for the subsequent analysis. Section 3 defines and examines the general properties of the $C^*$-extreme points of the $C^*$-convex set $\mathrm{UCP}(\mcl{A,\B{\mcl{H}}})$, where $\mcl{A}$ is a unital real $C^*$-algebra and $\mcl{H}$ is a real Hilbert space. In Example \ref{eg-complexi-not-Cstar}, we show that if $\Phi\in\mathrm{UCP}(\mcl{A,\B{\mcl{H}}})$ is a $C^*$-extreme point, its complexification may not necessarily be a $C^*$-extreme point. This example highlights the need to separately develop a theory for real $C^*$-algebra case to understand the $C^*$-extreme points. However, Proposition \ref{prop-Phi_c-Phi-C^*ext} shows that if the complexification is a $C^*$-extreme point, then the original UCP map is also a  $C^*$-extreme point, when $\mcl{H}$ is a finite-dimensional real Hilbert space.
    Though similar techniques from the complex case apply to the real case with only mild changes in the hypotheses, it is necessary to present them for further analysis. As in the complex case, Proposition \ref{prop-Cstar-ext-Lin-extr} establishes that $C^*$-extreme points of $\mathrm{UCP}(\mcl{A,\B{\mcl{H}}})$ are linear extreme points as well, when $\mcl{H}$ is a finite-dimensional real Hilbert space. Furthermore, in Theorem \ref{thm-Cext-abstract-char} we summarize several equivalent criteria for a UCP map to be a $C^*$-extreme point, which serve as the real analogues of known results from the complex case. Corollary \ref{cor-pure-and-homo} shows that pure UCP maps and $\ast$-homomorphisms are $C^*$-extreme points.  In the complex case, it is well-known (\cite{FaMo97}) that inflation of a pure state is always a $C^*$-extreme point. Consequently, in the complex case $C^*$-extreme points of $\mathrm{UCP}(\mcl{A,\B{\mcl{H}}})$ always exist. However, in the real $C^*$-algebra case, we do not yet know whether $\mathrm{UCP}(\mcl{A},\B{\mcl{H}})$ always contains $C^*$-extreme points. We provide a sufficient condition for an inflation of a pure state on a real $C^*$-algebra to be a $C^*$-extreme point in Proposition \ref{prop-state-inflation}. Corollary \ref{cor-state-inflation-matrix} guarantees the existence of $C^*$-extreme points of $\mathrm{UCP}(M_n(\mbb{R}),\B{\mcl{H}})$. Theorem \ref{thm-directsum-pure} states that any $C^*$-extreme point of $\mathrm{UCP}(\mcl{A},M_n(\mbb{R}))$ can be expressed as a direct sum of pure UCP maps up to unitary equivalence. A key ingredient in establishing this theorem is Proposition \ref{prop-irred-OS}. In the remainder of Section 3, we examine the necessary and sufficient conditions established in \cite[Theorem 2.1]{FaZh98} for a UCP map on complex $C^*$-algebra to be a $C^*$-extreme point, in the context of real $C^*$-algebras. We provide examples illustrating that the "if and only if" conditions do not hold in the real case. However in Theorems \ref{thm-FaZh-MAIN-realtype} and \ref{thm-MAIN-FaZh-realtype-converse}, under suitable hypothesis, we provide a real analogue of \cite[Theorem 2.1]{FaZh98}. Finally, in Section 4, we explore the $C^*$-extreme points of $\mathrm{UCP}(\mcl{A},\B{\mcl{H}})$, where $\mcl{A}$ is a commutative unital real $C^*$-algebra. In the complex case, it is known that $\ast$-homomorphisms are the only $C^*$-extreme points. However, this is not true in the case of real $C^*$-algebras; in Theorem \ref{thm-commutative-Cext-structure-MAIN} we completely characterize the $C^*$-extreme points for the case $\mcl{H}$ is finite-dimensional real Hilbert space. As an application of this characterization, Corollary \ref{cor-Cstar-ext-CSn} identifies the $C^*$-extreme points of the $C^*$-convex set of all contractive skew-symmetric real matrices in $M_n(\mbb{R})$.

\section{Preliminaries and basic results}\label{sec-Prelim}

    Unless stated otherwise, all vector spaces that we consider are defined over the field $\mbb{R}$ of real numbers. By a linear map, we always mean a real linear map. 
    A \textit{real $C^*$-algebra} is a real Banach $\ast$-algebra $\mcl{A}$ such that $\norm{a^*a}=\norm{a}^2$ and $1+a^*a$ is invertible in unit extension of the algebra for all $a\in\mcl{A}$. Any element $a\in\mcl{A}$ can be uniquely decomposed as $a=a_1+a_2$, where $a_1:=\frac{1}{2}(a+a^*)\in\mcl{A}$ is self-adjoint (i.e., $a_1^*=a_1$) and $a_2:=\frac{1}{2}(a-a^*)\in\mcl{A}$ is anti-symmetric (i.e., $a_2^*=-a_2$). Two real $C^*$-algebras $\mcl{A}$ and $\mcl{B}$ are said to be isomorphic, and write $\mcl{A}\cong\mcl{B}$, if there exists a $\ast$-isomorphism (i.e., a bijective $\ast$-homomorphism) between them. Such a $\ast$-isomorphism will be necessarily isometric. 
    
    Let $\mcl{H}$ be a real Hilbert space. Then the space $\B{\mcl{H}}$ of all bounded (real) linear maps on $\mcl{H}$ forms a unital real $C^*$-algebra; it is non-commutative if $\dim(\mcl{H})>1$. By the Gelfand-Naimark theorem (see \cite[Proposition 5.1.2]{Li03}), any real $C^*$-algebra $\mcl{A}$ is isometrically $\ast$-isomorphic to a norm-closed $\ast$-subalgebra of $\B{\mcl{H}}$ for some real Hilbert space $\mcl{H}$. Another example of a non-commutative unital real $C^*$-algebra is the algebra of quaternions, 
    \begin{align*}
        \mbb{H}:=\{\alpha+\beta \textbf{i}+\gamma \textbf{j}+\delta \textbf{k}: \alpha,\beta,\gamma,\delta\in\mbb{R}\},
    \end{align*}
    which is the four-dimensional real vector space with basis $\{1,\textbf{i},\textbf{j},\textbf{k}\}$ satisfying
        $\textbf{i}^2=\textbf{j}^2=\textbf{k}^2=-1=\textbf{ijk}$,
    where $1$ is the multiplicative identity. The involution $\ast$ and the $C^*$-norm $\norm{\cdot}$ on $\mbb{H}$ are given as follows: 
    \begin{align*}
        (\alpha+\beta \textbf{i}+\gamma \textbf{j}+\delta \textbf{k})^*&:=\alpha-\beta \textbf{i}-\gamma \textbf{j}-\delta \textbf{k} \\
        \norm{\alpha+\beta \textbf{i}+\gamma \textbf{j}+\delta \textbf{k}}&:=(\alpha^2+\beta^2+\gamma^2+\delta^2)^{\frac{1}{2}}.
    \end{align*}
    Now, if $\Omega$ is a compact Hausdorff topological space and '$-$' is a homeomorphism of $\Omega$ of period $2$, then the space
    \begin{align*}
         C(\Omega,-):=\{f:\Omega\to\mbb{C}: f \text{ is continuous and }f(\ol{w})=\ol{f(w)}\text{ for all }w\in \Omega\}
    \end{align*}
    is a unital commutative real $C^*$-algebra with involution defined by $f^*(w):=\ol{f(w)}=f(\ol{w})$ for all $w\in\Omega$. Moreover, every unital commutative real $C^*$-algebra is isomorphic to $C(\Omega,-)$ for some compact Hausdorff space $\Omega$  and a homeomorphism '$-$' of period 2 (c.f. \cite[Proposition 5.1.4]{Li03}). Note that if $-=id_\Omega$, the identity map on $\Omega$, then $C(\Omega,-)=C(\Omega,\mbb{R})$, the space of all real-valued continuous functions on $\Omega$. We let $C(\Omega)$ denote the space of all complex-valued continuous functions on $\Omega$. We refer to \cite{Li03,Goo82} for more details on real $C^*$-algebras. 

    Given any real vector space $\mcl{V}$, its complexification is denoted by $\mcl{V}_c=\mcl{V}+\hat{i}\mcl{V}$. If $T:\mcl{V}\to\mcl{W}$ is a linear map between real vector spaces, then its \textit{complexification} $T_c:\mcl{V}_c\to\mcl{W}_c$ is given as follows:
    \begin{align*}
     T_c(v_1+\hat{i}v_2):=T(v_1)+\hat{i}T(v_2),\qquad\forall~v_1,v_2\in\mcl{V}.
    \end{align*}
    If $\mcl{H}$ is a real Hilbert space, then its complexification $\mcl{H}_c:=\mcl{H}+\hat{i}~\mcl{H}$ is a complex Hilbert space with inner product given by
    \begin{align*}
     \ip{x_1+\hat{i}y_1,x_2+\hat{i}y_2}:=\ip{x_1,x_2}+\ip{y_1,y_2}+i(\ip{y_1,x_2}-\ip{x_1,y_2}).
    \end{align*}
    If $\{\mcl{H}_j\}_j$ is a family of real Hilbert spaces and $\mcl{H}=\bigoplus_j\mcl{H}_j$, then $\mcl{H}_c\cong\bigoplus_j(\mcl{H}_j)_c$ via a canonical isomorphism. Now, let $\mcl{A}$ be a real $C^*$-algebra. Then its complexification $\mcl{A}_c:=\mcl{A}+\hat{i}~\mcl{A}$ forms a $\ast$-algebra with addition, multiplication and involution $\ast$ given as follows: If $\lambda=s+it\in\mbb{C}$ and $a_j+\hat{i}b_j\in\mcl{A}_c$ for $j=1,2$, then 
    \begin{align*}
    \lambda(a_1+\hat{i}b_1)+(a_2+\hat{i}b_2):&=(sa_1-tb_1+a_2)+\hat{i}(sb_1+ta_1+b_2)\\
    (a_1+\hat{i}b_1)(a_2+\hat{i}b_2):&=(a_1a_2-b_1b_2)+\hat{i}(a_1b_2+b_1a_2)\\
    (a_1+\hat{i}b_1)^*:&= a_1^*-\hat{i}b_1^*.
    \end{align*}
    Furthermore, there exists a unique norm on $\mcl{A}_c$ with respect to which $\mcl{A}_c$ forms a complex $C^*$-algebra. This norm satisfies $\bnorm{a+\hat{i}0}=\norm{a}$ and $\bnorm{a+\hat{i}b}=\bnorm{a-\hat{i}b}$ for all $a,b\in\mcl{A}$. We consider $\mcl{A}\subseteq\mcl{A}_c$ by identifying $a\in\mcl{A}$ with $a+\hat{i}0\in\mcl{A}_c$. If $\mcl{A}=\B{\mcl{H}}$ for some real Hilbert space $\mcl{H}$, then $\mcl{A}_c\cong\B{\mcl{H}_c}$ via the $\ast$-isomorphism  $T_1+\hat{i}T_2\mapsto T$, where $T(x+\hat{i}y):=T_1(x)-T_2(y) + \hat{i}(T_1(y) + T_2(x))$ for all $T_1,T_2\in\B{\mcl{H}}$ and $x,y\in\mcl{H}$. In particular, if $\mcl{A}=M_n(\mbb{R})$, then $\mcl{A}_c\cong M_n(\mbb{C})$ in a natural way. 
    Now, suppose $\mcl{A}=C(\Omega,-)$ for some compact Hausdorff space $\Omega$ and a homeomorphism $-$ of $\Omega$ of period $2$. Then $\mcl{A}_c\cong C(\Omega)$ via the $\ast$-isomorphism  $f_1+\hat{i}f_2\mapsto f_1+if_2$; its inverse is given by $f\mapsto f_1+\hat{i}f_2$, where $f_1(w):=\frac{1}{2}(f(w)+\ol{f(\ol{\omega})})$ and $f_2(w):=\frac{1}{2i}(f(w)-\ol{f(\ol{w})})$ for all $w\in\Omega$. In particular, let $\Omega=\{1,2\}$ with $\ol{1}:=2$ and $\ol{2}:=1$. Then $C(\Omega,-)\cong\mbb{C}$ via the $\ast$-isomorphism  $f\mapsto f(1)$. Thus, if $\mcl{A}=\mbb{C}$, then from the above discussion, we have $\mcl{A}_c\cong \mbb{C}^2$ via the $\ast$-isomorphism $\lambda+\hat{i}\mu\mapsto(\lambda+i\mu,\ol{\lambda}+i\ol{\mu})$, where $\lambda,\mu\in\mbb{C}$; the inverse of this isomorphism is given by $(\lambda,\mu)\mapsto\frac{\lambda+\ol{\mu}}{2}+\hat{i}\frac{\lambda-\ol{\mu}}{2i}$. 

    Let $\mcl{A}$ be a unital real $C^*$-algebra and $\mcl{H}$ be a real Hilbert space. By a \textit{representation} of $\mcl{A}$ on $\mcl{H}$ we mean a unital $\ast$-homomorphism $\pi:\mcl{A}\to\B{\mcl{H}}$. In such a case, we let $\pi(\mcl{A})':=\{T\in\B{\mcl{H}}: T\pi(a)=\pi(a)T \text{ for all }a\in\mcl{A}\}$, and is called the commutant of $\pi$. A representation $\pi:\mcl{A}\to\B{\mcl{H}}$ is said to be \textit{irreducible} if $\mcl{H}$ does not have proper closed $\pi(\mcl{A})$-invariant subspaces. It is known (\cite{Li03},\cite{Ros16}) that if a representation $\pi:\mcl{A}\to\B{\mcl{H}}$ is irreducible, then $\pi(\mcl{A})'$ must be isomorphic to one of $\mbb{R},\mbb{C}$ or $\mbb{H}$. Following \cite{Ros16}, we say that an irreducible representation $\pi:\mcl{A}\to\B{\mcl{H}}$ is of \textit{real},  \textit{complex}, or \textit{quaternionic} type, depending on whether $\pi(\mcl{A})'$ is isomorphic to $\mbb{R,C}$ or $\mbb{H}$, respectively. Now, from \cite[Proposition 5.3.7]{Li03}, we have that $\pi$ is irreducible if and only if $\pi(\mcl{A})'\cap\B{\mcl{H}}_{sa}=\mbb{R}I_\mcl{H}$, where $\B{\mcl{H}}_{sa}$ is the set of self-adjoint operators on $\mcl{H}$. Note that if $\pi:\mcl{A}\to\B{\mcl{H}}$ is irreducible, then $\pi_c:\mcl{A}_c\to\B{\mcl{H}_c}$ is irreducible (i.e., $\pi_c(\mcl{A}_c)' = \mbb{C}I_{\mcl{H}_c}$) if and only if $\pi(\mcl{A})' = \mbb{R}I_{\mcl{H}}$ if and only if $\pi$ is of real type. 

    Let $\mcl{A}$ be a unital real $C^*$-algebra. An element $a\in\mcl{A}$ is said to be \textit{positive}, and write $a\geq 0$, if $a=b^*b$ for some $b\in\mcl{A}$. 
    Let $\mcl{H}$ be a real Hilbert space and $\Phi:\mcl{A}\to\B{\mcl{H}}$ be a linear map. Then $\Phi$ is said to be \textit{self-adjoint} if $\Phi(a^*)=\Phi(a)^*$ for all $a\in\mcl{A}$; \textit{positive} if $\Phi(a^*a)\geq 0$ for all $a\in\mcl{A}$; and it is said to be \textit{completely positive} (CP) if $\sum_{i,j=1}^nT_i^*\Phi(a_i^*a_j)T_j\geq 0$ for any finite subsets $\{a_j\}_{j=1}^n\subseteq\mcl{A}, \{T_j\}_{j=1}^n\subseteq\B{\mcl{H}}$ and $n\in\mbb{N}$. Observe that $\Phi$ is CP if and only if the map $\id_k\otimes\Phi:M_k(\mbb{R})\otimes \mcl{A}\to M_k(\mbb{R})\otimes\B{\mcl{H}}$ is a positive map for all $k\in\mbb{N}$, where $\id_k$ is the identity map on the matrix algebra $M_k(\mbb{R})$ of all $k\times k$ real matrices. Given a CP-map $\Phi : \mcl{A} \to \B{\mcl{H}}$ there exist a real Hilbert space $\mcl{K}$, a unital $\ast$-homomorphism $\pi : \mcl{A} \to \B{\mcl{K}}$ and a bounded linear operator $V \in \B{\mcl{H,K}}$ such that 
    $\Phi(a) = V^{*}\pi(a)V$ for all $a\in\mcl{A}$. Note that $\Phi$ is unital (i.e., $\Phi(1)=I$) if and only if $V$ is an isometry. Such a triple $(\mcl{K},\pi,V)$ is referred to as a \textit{Stinespring representation} of $\Phi$. We can always choose a triple that meets the \textit{minimality} condition, which is defined as $\mcl{K}=\cspan\{\pi(\mcl{A})V\mcl{H}\}$. Such a minimal Stinespring representation is unique up to unitary equivalence. See \cite{Sti55, Pau02, Ruan03} for details. Any unital self-adjoint positive linear functional $\phi:\mcl{A}\to\mbb{R}$ is called a \textit{real state}, and is necessarily a unital completely positive (UCP) map. In such cases, the operator $V:\mbb{R}\to\mcl{K}$ that appears in the Stinespring representation can be identified with the vector $z:=V(1)\in\mcl{K}$. Thus, the Stinespring representation  simplifies to the Gelfand-Naimark-Segal (GNS) construction $(\mcl{K},\pi,z)$ of $\phi$ (see \cite[Theorem 3.3.4]{Li03}). 
    
    We let $\mathrm{CP}(\mcl{A},\B{\mcl{H}})$ denote the space of all CP maps from $\mcl{A}$ to $\B{\mcl{H}}$.  Given CP-maps $\Phi,\Psi\in\mathrm{CP}(\mcl{A},\B{\mcl{H}})$ between real $C^*$-algebras, we say $\Psi$ is \it{dominated by} $\Phi$ (and write $\Psi\leq_{cp}\Phi$) if $\Phi-\Psi$ is also a CP-map. A CP map $\Phi:\mcl{A}\to\B{\mcl{H}}$ is said to be \it{pure} if the only CP maps $\Psi:\mcl{A}\to\B{\mcl{H}}$ for which $\Psi \leq_{cp} \Phi$ are those of the form $\Psi = t\Phi$ for some $t \in [0,1]$. The above definitions and results hold true even if $\mcl{A}$ is a complex $C^*$-algebra and $\mcl{H}$ is a complex Hilbert space. One can easily check that a linear map $\Phi:\mcl{A}\to\B{\mcl{H}}$ is CP if and only if its complexification  $\Phi_c:\mcl{A}_c\to\B{\mcl{H}}_c\cong\B{\mcl{H}_c}$ is CP; in particular, CP maps are necessarily self-adjoint (see \cite[Lemma 2.3]{BlTe21}). Furthermore, if $(\mcl{K},\pi,V)$ is a (minimal) Stinespring representation for $\Phi$, then $(\mcl{K}_c,\pi_c,V_c)$ is a (minimal) Stinespring representation for $\Phi_c$.

\begin{theorem}[\cite{Arv69}] \label{thm-radon-nikodym}
    Let $\Phi,\Psi:\mcl{A} \to \B{\mcl{H}}$ be CP maps, and let $(\mcl{K},\pi,V)$ be the minimal Stinespring representation of $\Phi$. Then the following hold:
    \begin{enumerate}[label=(\roman*)]
        \item $\Psi \leq_{cp} \Phi$ if and only if there exists a unique positive contraction $D \in \pi(\mcl{A})'$ such that $\Psi(\cdot) = V^*D\pi(\cdot)V$.
        \item $\Phi$ is pure if and only if $\pi$ is irreducible.
    \end{enumerate}  
\end{theorem}

    Arveson (\cite{Arv69}) proved the above theorem in the complex case. However, the same proof also works in the real case.

\begin{remark}\label{rmk-pure-maps}
    Let $\mcl{A}$ be a unital real $C^*$-algebra, $\mcl{H}$ be a real Hilbert space and $\Phi:\mcl{A}\to\B{\mcl{H}}$ be a CP map. 
    \begin{enumerate}[label=(\roman*)]
        \item If $\Phi_c$ is pure, then $\Phi$ is pure. For, let $\Psi\in \mathrm{CP}(\mcl{A},\B{\mcl{H}})$ be such that $\Psi\leq_{cp}\Phi$. Then $\Phi_c - \Psi_c = (\Phi-\Psi)_c$ so that $\Psi_c \leq_{cp} \Phi_c$. Hence there exists  $t \in [0,1]$ such that $\Psi_c = t\Phi_c$, and which implies that $\Psi = t\Phi$. Hence $\Phi$ is pure. 
        \item If $\Phi$ is pure, then $\Phi_c$ need not be pure in general. For example, in the real $C^*$-algebra $\mcl{A} = \mbb{C}$, the only real state is $\phi:\mcl{A} \to \mbb{R}$ defined by $\phi(\lambda) := Re \lambda$ for all $\lambda\in\mcl{A}$. But its complexification $\phi_c : \mbb{C}^2 \to \mbb{C}$ given by 
        \begin{align*}
         \phi_c(\lambda,\mu) 
            &= \phi_c\left(\frac{\lambda+\ol{\mu}}{2}+\hat{i}\frac{\lambda-\ol{\mu}}{2i}\right) = \phi\left(\frac{\lambda + \ol{\mu}}{2}\right) + i\phi\left(\frac{\lambda - \ol{\mu}}{2i}\right)\\
            &= Re\left(\frac{\lambda + \ol{\mu}}{2}\right) + iRe\left(\frac{\lambda - \ol{\mu}}{2i}\right) = \frac{\lambda+\mu}{2},\qquad\forall~\lambda,\mu\in\mbb{C},
        \end{align*}
        is not a pure state. 
        \item Let $\Phi$ be pure and $(\mcl{K},\pi,V)$ be the minimal Stinespring representation of $\Phi$. Then $\Phi_c$ is pure if and only if $\pi_c$ is irreducible if and only if $\pi$ is of real type. 
    \end{enumerate}
\end{remark}

    Any irreducible representation of $M_n(\mbb{R})$ is unitarily equivalent to the identity map on $M_n(\mbb{R})$. So, from the above theorem, it follows that any pure UCP map $\Phi\in\mathrm{UCP}(M_n(\mbb{R}),M_m(\mbb{R}))$ is of the form $\Phi(\cdot)=V^*(\cdot)V$ for some isometry $V\in M_{n\times m}(\mbb{R})$. For such maps, the complexification is $\Phi_c(\cdot)=V^*(\cdot)V$, a pure UCP map, where $V$ is viewed as a complex matrix. Thus, by Remark \ref{rmk-pure-maps}, we conclude that $\Phi\in\mathrm{UCP}(M_n(\mbb{R}),M_m(\mbb{R}))$ is pure if and only if $\Phi_c\in\mathrm{UCP}(M_n(\mbb{C}),M_m(\mbb{C}))$ is pure.

\section{General properties of $C^*$-extreme points}\label{sec-Genproperty-3}
    In this section, unless stated otherwise, let $\mcl{A}$ denote a unital real $C^*$-algebra, and assume that all Hilbert spaces under consideration are over the real field. The letters $\mcl{H}$ and $\mcl{K}$, possibly with subscripts, will be used to denote these Hilbert spaces. For a given operator $T\in\B{\mcl{H}}$ we define $\mathrm{Ad}_T$ on $\B{\mcl{H}}$ by $\mathrm{Ad}_T(X):=T^*XT$ for all $X\in\B{\mcl{H}}$.  We denote by $\mathrm{UCP}(\mcl{A,\B{\mcl{H}}})$ the set of all unital CP-maps from $\mcl{A}$ to $\B{\mcl{H}}$. Note that $\mathrm{UCP}(\mcl{A,\B{\mcl{H}}})$ is a \it{$C^*$-convex} set in the following sense: If $\Phi_j \in \mathrm{UCP}(\mcl{A},\B{\mcl{H}})$ and $T_j\in\B{\mcl{H}}$ for $1\leq j \leq n$ satisfy $\sum_{j=1}^n T_j^*T_j = I$, then their \it{$C^*$-convex combination} $\sum_{j=1}^n\mathrm{Ad}_{T_j}\circ\Phi_j$ is also an element of $\mathrm{UCP}(\mcl{A},\B{\mcl{H}})$. This $C^*$-convex combination is called \textit{proper} if each $T_j$ is invertible. Now following \cite{LoPa81, FaMo97}, we define the following: 

\begin{definition}
    A real linear map $\Phi\in \mathrm{UCP}(\mcl{A},\B{\mcl{H}})$ is called a 
    \begin{enumerate}[label=(\roman*)]
     \item \it{linear extreme point} if, whenever $\Phi= \sum_{j=1}^n t_j\Phi_j$ for some $t_j\in (0,1)$ satisfying $\sum_{j=1}^n t_j = 1 $, then $\Phi_j=\Phi$ for all $1\leq j\leq n$.
     \item \it{$C^*$-extreme point} of $\mathrm{UCP}(\mcl{A},\B{\mcl{H}})$ if, whenever $\Phi= \sum_{j=1}^n\mathrm{Ad}_{T_j}\circ\Phi_j$ for some invertible $T_j\in\B{\mcl{H}}$ satisfying $ \sum_{j=1}^n T_j^*T_j = I$, then there exist unitaries $U_j\in\B{\mcl{H}}$ such that $\Phi_j= \mathrm{Ad}_{U_j}\circ\Phi$ (and we write $\Phi_j\sim\Phi$) for all $1\leq j\leq n$.
    \end{enumerate}
\end{definition}

    Note that in the definition of the $C^*$-extreme point (resp. linear extreme point), it suffices to consider proper $C^*$-convex (resp. convex) combinations of two UCP maps (see \cite[Proposition 2.1.2]{Zho98}). We denote the set of all linear extreme points and $C^*$-extreme extreme points of $\mathrm{UCP}(\mcl{A},\B{\mcl{H}})$ by $\mathrm{UCP}_{ext}(\mcl{A},\B{\mcl{H}})$ and $\mathrm{UCP}_{C^*-ext}(\mcl{A},\B{\mcl{H}})$, respectively. Clearly, $\Phi \in \mathrm{UCP}_{C^*-ext}(\mcl{A},\B{\mcl{H}}) $ if and only if $\mathrm{Ad}_U\circ\Phi \in \mathrm{UCP}_{C^*-ext}(\mcl{A},\B{\mcl{H}})$ for any unitary $U\in\B{\mcl{H}}$. These definitions and notations also apply to complex $C^*$-algebras $\mcl{A}$ and complex Hilbert spaces $\mcl{H}$. Note that if $\dim(\mcl{H})=1$, i.e., for (real) states on $\mcl{A}$, the notions of pure, linear extreme and $C^*$-extreme coincide. 

    Now, let $\Phi\in\mathrm{UCP}(\mcl{A},\B{\mcl{H}})$. We first investigate whether $\Phi$ being a $C^*$-extreme point of $\mathrm{UCP}(\mcl{A},\B{\mcl{H}})$ implies, or is implied by, $\Phi_c$ being a $C^*$-extreme point of $\mathrm{UCP}(\mcl{A}_c,\B{\mcl{H}_c})$. 

\begin{example}\label{eg-complexi-not-Cstar}
    In general, $\Phi\in\mathrm{UCP}_{C^*-ext}(\mcl{A},\B{\mcl{H}})$ does not imply that $\Phi_c$ is a $C^*$-extreme point of $\mathrm{UCP}(\mcl{A}_c,\B{\mcl{H}_c})$. For example, consider the map $\phi:\mbb{C}\to\mbb{R}=\B{\mbb{R}}$ defined as in Remark \ref{rmk-pure-maps}. Being a pure real state $\phi\in\mathrm{UCP}_{C^*-ext}(\mbb{C},\B{\mbb{R}})$. But its complexification $\phi_c:\mbb{C}^2\to\mbb{C}=\B{\mbb{C}}$ is not pure, and hence not a $C^*$-extreme point.     
\end{example}   

\begin{lemma}\label{lem-unitary-orthogonal} 
    Let $\Psi,\Phi \in \mathrm{UCP}(\mcl{A},M_n(\mbb{R}))$ and $U\in M_n(\mbb{C})$ be unitary such that $\Psi(\cdot)=U^*\Phi(\cdot)U$. Then there exists a unitary $V\in M_n(\mbb{R})$ such that $\Psi(\cdot)=V^*\Phi(\cdot)V$.
\end{lemma}

\begin{proof}
    Given that $U\Psi(\cdot)= \Phi(\cdot)U$. Write $U=U_1+iU_2$ with $U_j\in M_n(\mbb{R})$, $j=1,2$. Then, we get $U_j\Psi(\cdot)= \Phi(\cdot)U_j$ for $j=1,2$. If any of the $U_j$ is unitary, we are done. If not, consider the non-zero polynomial $p:\mbb{C}\to\mbb{C}$ defined by $p(z) = \det(U_1+zU_2)$, which has only finitely many zeros. Choose and fix $\lambda\in\mbb{R}$ such that $S:= U_1+\lambda U_2$ is invertible. Note that $S\Psi(\cdot)=\Phi(\cdot)S$, and by taking adjoints on both sides, we get $\Psi(\cdot) S^*=S^*\Phi(\cdot)$. Let $S=VP$ be the polar decomposition of $S$, where $V\in M_n(\mbb{R})$ is a unitary and $P=(S^*S)^{\frac{1}{2}}\in M_n(\mathbb{R})$ is positive semidefinite. Observe that $\Psi(\cdot)S^*S = S^*\Phi(\cdot)S = S^*S\Psi(\cdot)$, which implies that $\Psi(\cdot)P=P\Psi(\cdot)$. Then, 
    \begin{align*}
        V\Psi(\cdot)P = VP\Psi(\cdot)=S\Psi(\cdot)=\Phi(\cdot)S=\Phi(\cdot)VP.
    \end{align*}
    Since $P$ is invertible, we get $V\Psi(\cdot)=\Phi(\cdot)V$, i.e.,  $\Psi=\mathrm{Ad}_V\circ\Phi$.
\end{proof}

\begin{proposition}\label{prop-Phi_c-Phi-C^*ext}
    Let $\Phi \in \mathrm{UCP}(\mcl{A},M_n(\mbb{R}))$ be given. If $\Phi_c\in \mathrm{UCP}_{C^*-ext}(\mcl{A}_c,M_n(\mbb{C}))$, then $\Phi\in\mathrm{UCP}_{C^*-ext}(\mcl{A},M_n(\mbb{R}))$. 
\end{proposition}

\begin{proof}
    Let $\Phi=\sum_{j=1}^2 \mathrm{Ad}_{T_j}\circ\Phi_j$ be a proper $C^*$-convex decomposition of $\Phi$. Then $\Phi_c=\sum_{j=1}^2 \mathrm{Ad}_{(T_j)_c}\circ(\Phi_j)_c$ is a proper $C^*$-convex decomposition of $\Phi_c$. By hypothesis, there exist unitaries $U_j \in M_n(\mbb{C})$ such that $\Phi_c = \mathrm{Ad}_{U_j}\circ(\Phi_j)_c$ for $j=1,2$. Considering $\mcl{A}\subseteq\mcl{A}_c$ via the canonical embedding, we get $\Phi(\cdot)=U_j^*\Phi_j(\cdot)U_j$. Then, by Lemma \ref{lem-unitary-orthogonal}, there exist unitary matrices $V_j \in M_n(\mbb{R})$ such that $\Phi=\mathrm{Ad}_{V_j}\circ\Phi_j$ for $j=1,2$.
\end{proof}    
    
    Many of the results discussed in the remainder of this section are known in the context of complex $C^*$-algebras. Since we are concerned with the real $C^*$-algebra counterparts of these results (with possibly minor modifications to the statements), we present them here. When the proofs from the complex case directly apply to the real case, we cite the relevant references for the complex case. If the proof does not transfer directly and requires technical adjustments, we either sketch or provide a full proof for completeness.

\begin{theorem}\label{thm-extreme-point-char}
    Let $\Phi\in\mathrm{UCP}(\mcl{A},\B{\mcl{H}})$ and $(\mcl{K},\pi,V)$ be its minimal Stinespring representation. Define $\mathscr{C} : \pi(\mcl{A})' \cap \B{\mcl{K}}_{sa} \to \B{\mcl{H}}$ by $\mathscr{C}(T) = V^*TV$ for all $T\in\B{\mcl{H}}$. Then $\Phi \in\mathrm{UCP}_{ext}(\mcl{A,\B{\mcl{H}}})$ if and only if $\mathscr{C}$ is an injection.
\end{theorem}

    The above theorem is the real case analogue of the abstract characterization of linear extreme points in the complex $C^*$-algebra case, due to Arveson (\cite[Theorem 1.4.6]{Arv69}), and the same proof works in the real $C^*$-algebra case as well.

\begin{lemma}\label{lem-compact-selfadj-comm}
    Let $\mcl{H}$ be a complex Hilbert space and $T_1, T_2 \in \B{\mcl{H}}$ be such that $T_1^*T_1$ and $T_2^*T_2$ are non-zero scalar multiples of the identity operator $I$ with $T_1^*T_1 + T_2^*T_2 = I$. 
    \begin{enumerate}[label=(\roman*)]
          \item If $S \in \B{\mcl{H}}$ is a compact self-adjoint operator such that $S = T_1^*ST_1 + T_2^*ST_2$, then $T_1S=ST_1$ and $T_2S=ST_2$.
          \item If $S \in \B{\mcl{H}}$ is a compact operator such that $S^* = -S$ and $S = T_1^*ST_1 + T_2^*ST_2$, then $T_1S=ST_1$ and $T_2S=ST_2$.
    \end{enumerate}
\end{lemma}

\begin{proof}
    The first part is exactly \cite[Lemma 2.2.1]{Zho98}. To prove the second part, note that $T:=iS$ is a compact self-adjoint operator. Also, $T= T_1^*TT_1 + T_2^*TT_2$. Then, by $(i)$, we have $T_1,T_2$ commutes with $T$ and therefore with $S$ as well.
\end{proof}
    
\begin{corollary}\label{cor-compact-selfadj-commut}
    Let $\mcl{H}$ be a real Hilbert space and $T_1,T_2\in\B{\mcl{H}}$ be such that $T_1^*T_1$ and $T_2^*T_2$ are non-zero scalar multiples of the identity operator, and $T_1^*T_1 + T_2^*T_2 = I$. If $S\in \B{\mcl{H}}$ is a compact self-adjoint (or compact anti-symmetric) operator such that $S = T_1^*ST_1 + T_2^*ST_2$, then $T_jS=ST_j$ for $j=1,2$.
\end{corollary}

\begin{proof}
    Since $(S_c)^* = (S^*)_c$, from the hypothesis it follows that $S_c\in\B{\mcl{H}_c}$ is a self-adjoint (or anti-symmetric) compact operator. Now, let $0 \neq \lambda_j \in \mbb{R}$ be such that $T_j^*T_j = \lambda_j I_{\mcl{H}}, j=1,2$. Then $ ({T_j})_c^*(T_j)_c=(T_j^*T_j)_c = \lambda_j I_{\mcl{H}_c}$. Also 
    \begin{align*}
        T_1^*T_1 + T_2^*T_2 = I_{\mcl{H}} &\implies (T_1)_c^*(T_1)_c + (T_2)_c^*(T_2)_c = I_{\mcl{H}_c}, \\
        S= T_1^*ST_1 + T_2^*ST_2 &\implies S_c = (T_1)_c^*S_c(T_1)_c + (T_2)_c^*S_c(T_2)_c.
    \end{align*}
    Then, by Lemma \ref{lem-compact-selfadj-comm}, we conclude that $S_c$ commutes with $(T_j)_c$ for $j=1,2$, and hence $S$ commutes with $T_j$ for $j=1,2$. 
\end{proof}

 
\begin{proposition}\label{prop-Cstar-ext-Lin-extr}
    Let $\Phi \in \mathrm{UCP}_{C^*-ext}(\mcl{A,B(H)})$. If $\Phi(\mcl{A}) \subseteq \K{\mcl{H}} + \mbb{R}\cdot I$, then $\Phi \in \mathrm{UCP}_{ext}(\mcl{A,B(H)})$. In particular, if $dim(\mcl{H}) < \infty$, then $\mathrm{UCP}_{C^*-ext}(\mcl{A,B(H)})\subseteq\mathrm{UCP}_{ext}(\mcl{A,B(H)})$.
\end{proposition}

\begin{proof}
    Let $\Phi = t_1\Phi_1+t_2\Phi_2$ for some $t_1,t_2\in (0,1)$ with $t_1+t_2=1$ and $\Phi_j \in \mathrm{UCP}(\mcl{A,B(H)}), j=1,2$. Then there exist unitaries $U_j \in \B{\mcl{H}}$ such that $\Phi_j = \mathrm{Ad}_{U_j}\circ\Phi$ for $j=1,2$. Let $T \in \Phi(\mcl{A})$ be a self-adjoint (or anti-symmetric) operator. Define $T_j:=\sqrt{t_j}U_j$ for $j=1,2$. Then
    \begin{align*}
        T=T_1^*TT_1 + T_2^*TT_2,\qquad T_j^*T_j=t_jI,\qquad\mbox{and}\qquad \sum_{j=1}^2T_j^*T_j=I.
    \end{align*}
    By assumption, $T=S+\lambda I$ for some $S \in \K{\mcl{H}}$ and $\lambda\in \mbb{R}$. Then, from the above, it follows that 
    $S=T_1^*ST_1 + T_2^*ST_2$. Also, $T$ is self-adjoint (or anti-symmetric), which implies that $S$ is self-adjoint (or anti-symmetric). So, by Corollary \ref{cor-compact-selfadj-commut}, $T_j$ commutes with $S$ and hence with $T$. As a result, $U_j$ commutes with every $T \in \Phi(\mcl{A})$ that is either self-adjoint or anti-symmetric. Therefore, $U_j$ commutes with all operators in $\Phi(\mcl{A})$, implying that $\Psi_j = \mathrm{Ad}_{U_j}\circ\Phi=\Phi$ for $j=1,2$. Thus, $\Phi \in \mathrm{UCP}_{ext}(\mcl{A,B(H)})$.  
\end{proof}

\begin{theorem}[\cite{FaZh98, Zho98, BhKu22}]\label{thm-Cext-abstract-char}
    Let $\Phi\in\mathrm{UCP}(\mcl{A},\B{\mcl{H}})$ and $(\mcl{K},\pi,V)$ be its minimal Stinespring representation. Then the following are equivalent:
    \begin{enumerate}[label=(\roman*)]
        \item $\Phi\in\mathrm{UCP}_{C^*-ext}(\mcl{A},\B{\mcl{H}})$; 
        \item For any positive operator $D \in \pi(\mcl{A})'$ with $V^*DV$ invertible, there exist a partial isometry $U \in \pi(\mcl{A})'$ with $\ran{U^*}=\ran{U^*U}=\clran{D^{\frac{1}{2}}}$ and an invertible operator $Z\in\B{\mcl{H}}$ such that $UD^{\frac{1}{2}}V=VZ$;
        \item For any positive operator $D \in \pi(\mcl{A})'$ with $V^*DV$ invertible, there exists $S \in \pi(\mcl{A})'$ such that $V^*SV$ is invertible (i.e., $S(V\mcl{H})\subseteq V\mcl{H}$ and $S|_{V\mcl{H}}$ is invertible), $SVV^*=VV^*SVV^*$ and $D=S^*S$;
        \item For any $\Psi\in\mathrm{CP}(\mcl{A},\B{\mcl{H}})$ with $\Psi \leq_{cp} \Phi$ and $\Psi(1)$ invertible, there exists an invertible operator $Z \in \B{\mcl{H}}$ such that $\Psi = \mathrm{Ad}_Z\circ\Phi$.
    \end{enumerate}
\end{theorem}

\begin{corollary}\label{cor-pure-and-homo}
    If $\Phi\in\mathrm{UCP}(\mcl{A},\B{\mcl{H}})$ is pure or multiplicative (i.e., $\ast$-homomorphism), then $\Phi\in \mathrm{UCP}_{C^*-ext}(\mcl{A},\B{\mcl{H}})$.     
\end{corollary}

\begin{proof}
    If $\Phi$ is pure, then the conclusion follows from the above theorem. Now, if $\Phi$ is a $\ast$-homomorphism, then as in \cite[Proposition 1.2]{FaMo97} we conclude that $\Phi$ is a $C^*$-extreme point. 
\end{proof}


\begin{definition}
    Let $\phi:\mcl{A}\to\mbb{R}$ be a real state. The map $\Phi : \mcl{A} \to \B{\mcl{H}}$ defined by 
    \begin{align*}
         \Phi(a):=\phi(a)I_{\mcl{H}},\qquad\forall~a\in\mcl{A},
    \end{align*}
    is called the \it{inflation} of $\phi$ on $\mcl{H}$.  
\end{definition}

\begin{proposition}\label{prop-state-inflation}
    Let $\phi:\mcl{A}\to\mbb{R}$ be a pure real state and let $(\mcl{K},\pi,z)$ be its minimal GNS representation with the irreducible representation $\pi$ is of real type. If $\Phi$ is the inflation of $\phi$ on $\mcl{H}$, then $\Phi\in\mathrm{UCP}_{C^*-ext}(\mcl{A},\B{\mcl{H}})$.
\end{proposition}

\begin{proof}
    Define $\widetilde{\pi}:\mcl{A}\to\B{\mcl{H}\otimes\mcl{K}}$ and $\widetilde{V}:\mcl{H}\otimes\mbb{R}\to\mcl{H}\otimes\mcl{K}$ by
    \begin{align*}
        \widetilde{\pi}(\cdot):=I_{\mcl{H}}\otimes\pi(\cdot)
        \qquad\text{and}\qquad
        \widetilde{V}:=I_{\mcl{H}}\otimes V,
    \end{align*}
    where $V(\lambda):=\lambda z$ for all $\lambda\in\mbb{R}$. We identify $\mcl{H}=\mcl{H}\otimes\mbb{R}$ via the isomorphism $x\mapsto x\otimes 1$, and observe that $\Phi=\mathrm{Ad}_{\widetilde{V}}\circ\widetilde{\pi}$. Thus, $(\mcl{H}\otimes\mcl{K},\widetilde{\pi},\widetilde{V})$ is the minimal Stinespring representation of $\Phi$. Now, let $\Psi\in\mathrm{CP}(\mcl{A},\B{\mcl{H}})$ be such that $\Psi\leq_{cp}\Phi$ and $\Psi(1)$ invertible.  Then, by Theorem \ref{thm-radon-nikodym}, there exists a positive contraction $\widetilde{D}\in {\widetilde{\pi}}(\mcl{A})'\subseteq\B{\mcl{H}\otimes\mcl{K}}$ such that $\Psi(\cdot)=\widetilde{V}^*\widetilde{D}\widetilde{\pi}(\cdot)\widetilde{V}$. Since $\pi(\mcl{A})' = \mbb{R} I_{\mcl{K}}$, from \cite[Theorem 4.4.3]{Li03}, it follows that $\widetilde{\pi}(\mcl{A})'=\B{\mcl{H}} \otimes I_{\mcl{K}}$.
    So, there exists $D\in\B{\mcl{H}}$ such that $\widetilde{D}= D\otimes I_{\mcl{K}}$. Now, $0 \leq \widetilde{D} \leq I_{\mcl{H}\otimes\mcl{K}}$ implies $0\leq D \leq I_{\mcl{H}}$. Therefore,
    \begin{align*}
        \Psi(a)
            &=(I_{\mcl{H}}\otimes V)^*(D\otimes I_\mcl{K})(I_\mcl{H}\otimes\pi(a))(I_{\mcl{H}}\otimes V)\\
            &= D \otimes (V^*\pi(a)V) 
             = D \otimes \phi(a) 
            = \phi(a)D\qquad\qquad\qquad(\because~\mcl{H}=\mcl{H}\otimes\mbb{R})\\
            &= \mathrm{Ad}_{D^{\frac{1}{2}}}\circ \Phi(a), \qquad\forall~a\in\mcl{A},
    \end{align*}
       where $D^{\frac{1}{2}}$ is invertible as $D=\Psi(1)$ is invertible.  Hence, by Theorem \ref{thm-Cext-abstract-char}, we conclude that $\Phi \in \mathrm{UCP}_{C^*-ext}(\mcl{A},\B{\mcl{H}})$.
\end{proof}

    Since every irreducible representation of a matrix algebra $M_n(\mbb{R})$ is unitarily equivalent to the identity map, which is of real type, we have the following. 

\begin{corollary}\label{cor-state-inflation-matrix}
    Let $\phi:M_n(\mbb{R})\to\mbb{R}$ be a pure real state and $\mcl{H}$ be a real Hilbert space. Then $\Phi(\cdot):=\phi(\cdot)I_{\mcl{H}}$ is a $C^*$-extreme point of $\mathrm{UCP}(M_n(\mbb{R}),\B{\mcl{H}})$.
\end{corollary}

    In contrast to the complex $C^*$-algebra case, the inflation of a pure real state is not necessarily a $C^*$-extreme point.

\begin{example}\label{eg-inflation-complex}
    Consider the pure real state $\phi:\mathbb{C} \to \mathbb{R}$ given by $\phi(\lambda):=Re(\lambda)$ for all $\lambda\in\mbb{C}$. Let $\Phi:\mathbb{C}\to\B{\mathbb{R}^2}=M_2(\mathbb{R})$ be the inflation of $\phi$, i.e., 
    \begin{align*}
         \Phi(\lambda) = \Matrix{Re(\lambda) & 0 \\0 & Re(\lambda)},\qquad\forall~\lambda\in\mbb{C}.
    \end{align*}
    Observe that $\Phi(\lambda)=\frac{1}{2}\Phi_1(\lambda) + \frac{1}{2}\Phi_2(\lambda)$,  where $\Phi_1,\Phi_2:\mbb{C}\to M_2(\mathbb{R})$ are the UCP maps (in fact, $\ast$-homomorphisms) given by 
    \begin{align*}
        \Phi_1(\lambda)=\Matrix{Re(\lambda) & Im(\lambda) \\
          -Im(\lambda) & Re(\lambda)}
        \qquad\mbox{and}\qquad
        \Phi_2(\lambda)=\Matrix{Re(\lambda) & -Im(\lambda) \\
        Im(\lambda) & Re(\lambda)}.
    \end{align*}       
    Thus, $\Phi$ is not a linear extreme and hence not a $C^*$-extreme either. Note that $\phi$ has the minimal Stinespring representation $(\mbb{R}^2,\pi,V)$, where $V:=\big[1~~0\big]^*\in M_{2\times 1}(\mbb{R})$ and the representation $\pi:=\Phi_1$ is of complex type. (Given an $m\times n$ real matrix $A$, we use $A^*$ to denote the transpose of $A$.)
\end{example}

\begin{example}\label{eg-inflation-quaternion}
    Let $\Phi:\mathbb{H}\to\B{\mathbb{R}^2}=M_2(\mathbb{R})$ be the inflation of the pure real state $\phi:\mathbb{H} \to \mathbb{R}$ given by $\phi(\alpha+\beta\textbf{i}+\gamma\textbf{j}+\delta\textbf{k}):=\alpha$, that is, 
    \begin{align*}
         \Phi(\alpha+\beta\textbf{i}+\gamma\textbf{j}+\delta\textbf{k}) = \Matrix{\alpha & 0 \\0 &\alpha}.
    \end{align*}
    Observe that $\Phi=\frac{1}{2}\Phi_1 + \frac{1}{2}\Phi_2$,  where $\Phi_1,\Phi_2:\mbb{H}\to M_2(\mathbb{R})$ are the UCP maps given by 
    \begin{align*}
        \Phi_1(\alpha+\beta\textbf{i}+\gamma\textbf{j}+\delta\textbf{k}):=\Matrix{\alpha & \beta \\-\beta & \alpha}
        \qquad\mbox{and}\qquad
        \Phi_2(\alpha+\beta\textbf{i}+\gamma\textbf{j}+\delta\textbf{k}):=\Matrix{\alpha & -\beta \\ \beta & \alpha}.
    \end{align*}       
    Thus, $\Phi$ is not a linear extreme and hence not a $C^*$-extreme either. Note that $\phi$ has the minimal Stinespring representation $(\mbb{R}^4,\pi,V)$, where $V=\big[1~0~0~0\big]^*$ and $\pi:\mbb{H}\to\B{\mbb{R}^4}=M_4(\mbb{R})$ is given by 
    \begin{align}\label{eq-piH}
        \pi(\alpha+\beta\textbf{i}+\gamma\textbf{j}+\delta\textbf{k}):=\Matrix{\alpha&-\beta&-\gamma&-\delta\\
                   \beta&\alpha&-\delta&\gamma\\
                   \gamma&\delta&\alpha&-\beta\\
                   \delta&-\gamma&\beta&\alpha
                  },
    \end{align}
    which is an irreducible representation of quaternionic type. 
\end{example}

\begin{notation}
    Given a countable family $\{\Phi_j\}_j$ of UCP maps $\Phi_j:\mcl{A}\to\B{\mcl{H}_j}$ we define $\bigoplus_j\Phi_j:\mcl{A}\to\B{\bigoplus_j\mcl{H}_j}$ by
    \begin{align*}
        (\bigoplus_j\Phi_j)(a):=\bigoplus_j\Phi_j(a),\qquad\forall~a\in\mcl{A},
    \end{align*}
    which is again a UCP map. The above examples also illustrate that the direct sum of $C^*$-extreme points need not be a $C^*$-extreme point.
\end{notation}    

\begin{proposition}\label{prop-CompExtremepoint} 
    For each $j=1,2$, let $\mcl{H}_j$ be a finite-dimensional and $\Phi_j\in\mathrm{UCP}(\mcl{A},\B{\mcl{H}_j})$. If $\Phi_1\oplus\Phi_2\in\mathrm{UCP}_{C^*-ext}(\mcl{A},\B{\mcl{H}_1\oplus\mcl{H}_2})$, then $\Phi_j\in \mathrm{UCP}_{C^*-ext}(\mcl{A},\B{\mcl{H}_j})$ for $j=1,2$.
\end{proposition}

The above proposition can be proved along similar lines to the proof of \cite[Lemma 4.1]{FaZh98}, together with Lemma \ref{lem-unitary-orthogonal}. For completeness, we include the proof below.

\begin{proof}
    Let $\mcl{H}:=\mcl{H}_1\oplus\mcl{H}_2$ and $\Phi:=\Phi_1\oplus\Phi_2$. Suppose $\Phi_1 = \sum_{j=1}^k\mathrm{Ad}_{T_j}\circ\Psi_j$  is a proper $C^*$-convex decomposition, where $T_j\in\B{\mcl{H}_1}$ invertible and $\Psi_j\in\mathrm{UCP}(\mcl{A},\B{\mcl{H}_1})$. Let $\wtilde{T}_j:=T_j\oplus\frac{1}{\sqrt{k}}I_{\mcl{H}_2}\in\B{\mcl{H}}$ and $\wtilde{\Psi}_j:=\Psi_j\oplus\Phi_2\in\mathrm{UCP}(\mcl{A},\B{\mcl{H}})$ for all $1\leq j\leq k$. Then  
        $\Phi=\sum_{j=1}^k\mathrm{Ad}_{\wtilde{T}_j}\circ\wtilde{\Psi}_j$
    is a proper $C^*$-convex decomposition of $\Phi$. Since $\Phi\in\mathrm{UCP}_{C^*-ext}(\mcl{A},\B{\mcl{H}})$, there exist unitaries $W_j\in\B{\mcl{H}}$ such that $\Phi=\mathrm{Ad}_{W_j}\circ\wtilde{\Psi}_j$, i.e., $\Phi_1\oplus\Phi_2 =\mathrm{Ad}_{W_j}\circ(\Psi_j\oplus\Phi_2)$ for all $1\leq j\leq k$. Now, let $a_1,\cdots,a_n\in\mcl{A}$ be such that $\{\Phi(a_1),\cdots,\Phi(a_n) \}$ is a basis of the finite-dimensional space $\Phi(\mcl{A})$. Fix $1\leq j\leq k$. We identify $\B{\mcl{H}_1} \cong M_{m}(\mbb{R})$, where $m=\dim(\mcl{H}_1)$. Now, we view the matrices in $M_m(\mbb{R})$ as complex matrices. Then, as in \cite[Lemma 4.1]{FaZh98}, using Specht's theorem \cite{Wie61} (for finite sets of matrices), we get a unitary operator $U_j \in M_m(\mbb{C})$ such that $\Phi_1(a_i)=U_j^*\Psi_j(a_i)U_j$ for each $1\leq i\leq  n$. Moreover, for each $a\in\mcl{A}$, there exist uniquely determined scalars $\alpha_i(a)\in\mbb{R}$ such that $\Phi(a)=\sum_{i=1}^n\alpha_i(a) \Phi(a_i)$. Then
    \begin{align*}
        \Phi_1(a)=\sum_{i=1}^n\alpha_i(a) \Phi_1(a_i)= U_j^*\Psi_j(\sum_{i=1}^n\alpha_i(a)a_i)U_j,\qquad\forall~a\in\mcl{A}.
    \end{align*}
    Now, given $a\in\mcl{A}$, let $b=a-\sum_{i=1}^n\alpha_i(a)a_i$. Note that $\Phi_1(b)=0$. Since $\Phi(b)$ and $\wtilde{\Psi}_j(b)$ are unitarily equivalent, they must have the same rank. Hence, we conclude that $\Psi_j(b)=0$,  i.e., $\Psi_j(\sum_{i=1}^n \alpha_i(a) a_i) = \Psi_j(a)$. Therefore, from the above equation, we obtain $\Phi_1(a) = U_j^*\Psi_j(a)U_j$ for all $a\in\mcl{A}$. Consequently, by Lemma \ref{lem-unitary-orthogonal}, there exist unitary matrices $V_j\in M_m(\mbb{R})$ such that $\Phi_1 = \mathrm{Ad}_{V_j}\circ\Psi_j$ for each $1\leq j\leq k$. Thus, $\Phi_1$ is $C^*$-extreme. Similarly, we can prove that $\Phi_2$ is also a $C^*$-extreme point.
\end{proof}

In the remainder of this section, we analyze the structure of $C^*$-extreme points. As a first step, we present the following technical proposition, which serves as a key ingredient in the proof of one of the main theorems that follows. To that end, we begin by recalling some relevant terminology. Let $\mcl{S} \subseteq\B{\mcl{H}}$ be a self-adjoint family of operators (i.e., $\mcl{S} = \mcl{S}^*$). Then the commutant $\mcl{S}':=\{T\in\B{\mcl{H}}: TZ=ZT\mbox{ for all }Z\in \mcl{S}\}$ of $\mcl{S}$ is a real von Neumann algebra \cite[Proposition 4.3.2 (1)]{Li03}. We say $\mcl{S}$ is \it{irreducible} if the only closed subspaces of $\mcl{H}$ which are invariant under $\mcl{S}$ are $\{0\}$ or $\mcl{H}$. Equivalently, the only projections in $\mcl{S}'$ are $0$ or $I_{\mcl{H}}$. In view of \cite[Proposition 4.3.4 (3)]{Li03}, $\mcl{S}$ is irreducible if and only if $\mcl{S}' \cap \B{\mcl{H}}_{sa} = \mbb{R} I_{\mcl{H}}$. We say a self-adjoint real linear map $\Phi:\mcl{A} \to \B{\mcl{H}}$ is \it{irreducible} if $\Phi(\mcl{A})$ is irreducible. 

\begin{proposition}\label{prop-irred-OS}
    Let $\mcl{H}$ be a finite-dimensional real Hilbert space, $\mcl{S}\subseteq\B{\mcl{H}}$ be an irreducible self-adjoint family of operators and $W_j\in\B{\mcl{H}},j=1,2$ be such that $\sum_{j=1}^2W_j^*ZW_j=Z$ for all $Z\in\mcl{S}$ and $\sum_{j=1}^2W_j^*W_j=I$. Then $W_jZ=ZW_j$ for all $j=1,2$ and $Z\in C^*(\mcl{S})$, the $C^*$-algebra generated by $\mcl{S}$.
\end{proposition}

\begin{proof}
    Consider the UCP map $\Psi:\B{\mcl{H}}\to\B{\mcl{H}}$ defined by 
    \begin{align*}
        \Psi(Z)=W_1^*ZW_1+W_2^*ZW_2,\qquad\forall~Z\in\B{\mcl{H}}.
    \end{align*}
    From assumption, $\Psi(Z)=Z$ for all $Z\in\mcl{S}$. Note that $(\mcl{H}\oplus\mcl{H},\pi,V)$ is a Stinespring representation of $\Psi$, where 
    \begin{align*}
        V=W_1\oplus W_2,
        \quad\mbox{and}\quad
        \pi=\id_{\B{\mcl{H}}} \oplus\id_{\B{\mcl{H}}}.
    \end{align*}
    Let $\mcl{H}_0$ be a minimal non-zero subspace of $\mcl{H}$ satisfying $W_j(\mcl{H}_0) \subseteq \mcl{H}_0$, $j=1,2$. Note that $\mcl{H}_0$ can be $\mcl{H}$ itself. Consider 
    \begin{align*}
        \Gamma:= \{Z \in \B{\mcl{H}}: W_jZh = ZW_jh , \forall h\in\mcl{H}_0, j=1,2\}
    \end{align*}  
    Clearly, $\Gamma$ is a subspace of $\B{\mcl{H}}$ which contains $I_{\mcl{H}}$. Now, fix $Z \in \Gamma$ and $S\in\mcl{S}$. Define 
    \begin{align*}
        \mcl{H}_0':=\{h\in\mcl{H}_0 : \norm{SZh} = \norm{SZ|_{\mcl{H}_0}} \norm{h}\}
    \end{align*}
    Since operators on finite-dimensional space attains norm, $\mcl{H}_0'$ contains non-zero vectors. Now, if $h\in\mcl{H}_0'$, then
    \begin{align*}
        \norm{SZh}^2 
        &= \norm{\Psi(S)Zh}^2 = \norm{W_1^*SW_1Zh + W_2^*SW_2Zh}^2 \\
        &= \norm{W_1^*SZW_1h + W_2^*SZW_2h}^2 = \norm{\Psi(SZ)h}^2 \\
        &= \norm{V^*\pi(SZ)Vh}^2 = \norm{VV^*\pi(SZ)Vh}^2 \\ 
        &\leq \norm{\pi(SZ)Vh}^2 = \norm{SZW_1h}^2 + \norm{SZW_2h}^2 \\ 
        &\leq \norm{SZ|_{\mcl{H}_0}}^2 \left( \norm{W_1h}^2  + \norm{W_2h}^2 \right) \\ 
        &= \norm{SZ|_{\mcl{H}_0}}^2 \norm{h}^2 = \norm{SZh}^2.
    \end{align*}
    Since both sides of the above inequality are the same, the inequalities becomes equalities and in particular we have  
    \begin{align}
        &\norm{SZW_jh}^2 = \norm{SZ|_{\mcl{H}_0}}^2 \norm{W_jh}^2,\qquad\forall~h\in\mcl{H}_0'\quad\mbox{ and}, \label{eq.1}\\
        &\norm{VV^*\pi(SZ)Vh}^2 = \norm{\pi(SZ)Vh}^2,\qquad\forall~h\in\mcl{H}_0'.\label{eq.2}
    \end{align}
    From \eqref{eq.1} it follows that $W_j(\mcl{H}_0') \subseteq \mcl{H}_0'$, $j=1,2$. By the minimality of $\mcl{H}_0$, we must have $\mcl{H}_0 = \lspan{\mcl{H}_0'}$. Now, since $VV^*$ is an orthogonal projection, from \eqref{eq.2}, we conclude that $VV^*\pi(SZ)Vh = \pi(SZ)Vh$ for all $h\in \mcl{H}_0'$; expanding the left- and right-hand sides we get
    \begin{align*}
        SZW_1h &= W_1W_1^*SZW_1h+W_1W_2^*SZW_2h\\
        SZW_2h &= W_2W_1^*SZW_1h+W_2W_2^*SZW_2h
    \end{align*}
    for all $h\in\mcl{H}_0'$. Then, since $Z\in\Gamma$, we get 
    \begin{align*}
        SZW_1h 
                = W_1W_1^*SW_1Zh + W_1W_2^*SW_2Zh 
                = W_1\Psi(S)Zh
                = W_1SZh
    \end{align*}
 for all $h\in\mcl{H}_0'$. Similarly, we get $SZW_2h = W_2SZh$ for all $h\in\mcl{H}_0'$. Now, since $\mcl{H}_0 = \lspan{\mcl{H}_0'}$, we conclude that $SZW_jh = W_jSZh$ for all $h\in \mcl{H}_0$, $j=1,2$.  Therefore, $SZ \in \Gamma$ for any $Z\in\Gamma$ and $S\in\mcl{S}$. Putting $Z=I_\mcl{H}$, we get $\mcl{S}\subseteq\Gamma$. Thus, if $S,Z\in\mcl{S}$, then $SZ\in\Gamma$. This implies $C^*(\mcl{S})\subseteq \Gamma$.  That is, we have $W_jZh = ZW_jh$ for all $Z\in C^*(\mcl{S})$ and $h\in\mcl{H}_0$. Now, fix $0\neq h_0 \in \mcl{H}_0$. Let $Z\in C^*(\mcl{S})$ and $h\in\mcl{H}$ be arbitrary. Note that since $\mcl{S}$ is irreducible, $C^*(\mcl{S})$ is an irreducible $C^*$-subalgebra of $\B{\mcl{H}}$. Then, by the transitivity property of $C^*(\mcl{S})$ \cite[Theorem 5.3.10]{Li03}, there exists $Y\in C^*(\mcl{S})$ such that $Yh_0 = h$. Then we use the commutator identity \[ [W_j,Z]Yh_0 = [W_j,ZY]h_0 - Z[W_j,Y]h_0\] to conclude that $[W_j,Z]h = 0$. That is, $W_jZ=ZW_j$ for all $Z\in C^*(\mcl{S})$.
\end{proof}

    The complex $C^*$-algebra version of the following theorem has been proved in several works in the literature (\cite{Far00,FaMo97,Zho98}), with each proof involving non-trivial algebraic computations. 
    In this paper, we present a proof of the real version of the theorem, drawing on ideas from \cite{Far00,Zho98}, with the above proposition serving as the main technical ingredient.



\begin{theorem}[{\cite[Theorem 2.1]{FaMo97}}] \label{thm-directsum-pure} 
 Let $\mcl{A}$ be a unital real $C^*$-algebra, $\mcl{H}$ be a finite-dimensional real Hilbert space and $\Phi\in\mathrm{UCP}_{C^*-ext}(\mcl{A},\B{\mcl{H}})$. Then there exist finitely many closed subspaces $\mcl{H}_i \subseteq \mcl{H}$ and maps $\Phi_i \in \mathrm{UCP}(\mcl{A},\B{\mcl{H}_i})$ such that 
 \begin{enumerate}[label=(\roman*)]
     \item $\Phi_i$ is pure for each $1\leq i\leq k$, and
     \item $\mcl{H} = \oplus_i \mcl{H}_i$ and $\Phi \sim \oplus_i \Phi_i$.
 \end{enumerate}
 That is, $\Phi$ is unitarily equivalent to a direct sum of pure UCP maps.
\end{theorem}

\begin{proof}
    We divide the proof into steps.\\
    \ul{Step 1:} We show that if $\Phi$ is irreducible, then it is pure. So, let $(\mcl{K},\pi,V)$ be the minimal Stinespring representation of $\Phi.$ Since $\Phi$ is $C^*$-extreme, it is also an extreme point by Proposition \ref{prop-Cstar-ext-Lin-extr}. Thus, it suffices to show that the injective map $\msc{C}:\pi(\mcl{A})' \cap \B{\mcl{K}}_{sa}\to\B{\mcl{H}}$, from Theorem \ref{thm-extreme-point-char}, has one-dimensional range. Indeed, in that case, the domain of $\msc{C}$ must also be one-dimensional, i.e., $\pi(\mcl{A})' \cap \B{\mcl{K}}_{sa} = \mbb{R}I_{\mcl{K}}$. Then, by \cite[Proposition 5.3.7]{Li03}, it follows that $\pi$ is irreducible, and hence $\Phi$ is pure. Now, to prove that the range of $\msc{C}$ is one-dimensional, let $D \in \pi(\mcl{A})'$ be a positive operator. For $\epsilon>0$, set $D_{\epsilon} := D+\epsilon I_{\mcl{K}}$. Then $V^*D_{\epsilon}V$ is invertible in $B(\mcl{H})$. Choose and fix a scalar $t>0$ such that $\norm{tD_{\epsilon}} <1$. Then $\norm{tV^*D_{\epsilon}V} < 1$. Let $D_1:=tD_{\epsilon}$ and $D_2:=I_{\mcl{K}}-tD_{\epsilon}$.  Note that $V^*D_1V$ and $V^*D_2V$ are positive invertible operators in $\B{\mcl{H}}$. Then $T_j:=(V^*D_jV)^{1/2}, j=1,2$, are invertible operators such that $T_1^*T_1+T_2^*T_2=I_\mcl{H}$, and $\Phi=\sum_{j=1}^{2} \mathrm{Ad}_{T_j}\circ\Phi_j$ is a proper $C^*$-convex combination of the UCP maps 
    \begin{align*}
        \Phi_j(\cdot):= T_j^{-1}V^*D_j\pi(\cdot)VT_j^{-1}, \qquad~j=1,2.
    \end{align*}
    Since $\Phi$ is $C^*$-extreme, there exist unitaries $U_j \in \B{\mcl{H}}$ such that $\Phi_j=\mathrm{Ad}_{U_j}\circ\Phi$, $j=1,2$. Letting $W_j := U_jT_j\in\B{\mcl{H}}, j=1,2$, we observe that $\sum_{j=1}^2 W_j^*W_j=I_\mcl{H}$ and 
    \begin{align*}
        \Phi = \sum_{j=1}^{2} \mathrm{Ad}_{W_j}\circ\Phi.
    \end{align*}
    Now consider the UCP map $\Psi:\B{\mcl{H}} \to \B{\mcl{H}}$ defined by 
    \begin{align*}
        \Psi(Z):= W_1^*ZW_1 + W_2^*ZW_2,\qquad\forall~Z\in\B{\mcl{H}}.
    \end{align*}
    From the construction we have $\Psi(Z)=Z$ for all $Z\in\Phi(\mcl{A})$. Then, taking $\mcl{S}=\Phi(\mcl{A})$ in Proposition \ref{prop-irred-OS}, we have $W_jZ=ZW_j$ for all $Z\in C^*(\Phi(\mcl{A}))$, $j=1,2$. Hence $W_j^*W_j$ is also in the commutant of $C^*(\Phi(\mcl{A}))$ for $j=1,2$. Note that since $\Phi$ is irreducible, $C^*(\Phi(\mcl{A}))$ is an irreducible $C^*$-subalgebra of $\B{\mcl{H}}$. Hence by \cite[Proposition 5.3.7]{Li03}, $ V^*D_jV=W_j^*W_j \in \mbb{R}_+ I_{\mcl{H}}$. It follows that $V^*D_{\epsilon}V \in\mbb{R}_+ I_{\mcl{H}}$, and consequently $V^*DV\in\mbb{R}_+ I_{\mcl{H}}$. Thus, $V^*DV \in \mbb{R}_+ I_{\mcl{H}}$ for all positive operators $D\in\pi(\mcl{A})'$, and hence $V^*DV\in\mbb{R}I_{\mcl{H}}$ for all $D\in\pi(\mcl{A})' \cap \B{\mcl{K}}_{sa}$. Therefore, we conclude that the range of $\msc{C}$ is one-dimensional and hence $\Phi$ is pure.\\
    \ul{Step 2:} Suppose $\Phi$ is not irreducible. Then there exists a non-trivial projection $P\in \Phi(\mcl{A})'$. Since $\mcl{H}$ is finite-dimensional, without loss of generality, we assume that $P$ is a minimal projection in $\Phi(\mcl{A})'$. Set $\mcl{H}_1 = \ran{P}$. Then, $\Phi$ can be decomposed into a direct sum of two UCP maps, namely $\Phi = \Phi_1 \oplus \wtilde{\Psi}$, where 
    \begin{align*}
        \Phi_1(\cdot) :&= \Phi(\cdot)|_{\mcl{H}_1} \in \mathrm{UCP}(\mcl{A},\B{\mcl{H}_1}),\quad\mbox{and}\\
        \widetilde{\Psi}(\cdot) :&= \Phi(\cdot)|_{\mcl{H}_1^{\perp}} \in \mathrm{UCP}(\mcl{A},\B{\mcl{H}_1^{\perp}})
    \end{align*}
    Then, by Proposition \ref{prop-CompExtremepoint}, we have $\Phi_1 \in \mathrm{UCP}_{C^*-ext}(\mcl{A},\B{\mcl{H}_1})$ and $\widetilde{\Psi} \in \mathrm{UCP}_{C^*-ext}(\mcl{A},\B{\mcl{H}_1^{\perp}})$. Note that $\Phi_1$ is irreducible due to the minimality of $P$. Hence, by Step 1, $\Phi_1$ is pure. Now, if $\widetilde{\Psi}$ is irreducible (and hence pure) we are done. Otherwise, repeat the above argument for the $C^*$-extreme point $\widetilde{\Psi}$ of $\mathrm{UCP}(\mcl{A},\B{\mcl{H}_1^{\perp}})$, where $\mcl{H}_1^{\perp} \subset \mcl{H}$ is a proper subspace. Since $\dim{\mcl{H}} < \infty$, after finitely many (say $k$) iterations of this process, we obtain real Hilbert spaces $\{\mcl{H}_j\}_{j=1}^k$ and pure UCP maps $\{\Phi_j:\mcl{A} \to \B{\mcl{H}_j}\}_{j=1}^k$ with $\mcl{H} = \oplus_{j=1}^k \mcl{H}_j$ and $\Phi = \oplus_{j=1}^k \Phi_j$.
\end{proof}

    Given Hilbert spaces $\mcl{K}_1, \mcl{K}_2$  and two representations $\pi_j:\mcl{A}\to\B{\mcl{K}_j}, j=1,2$, we define their intertwining space as 
    \begin{align*}
        \mcl{I}(\pi_1,\pi_2):=\{T\in\B{\mcl{K}_1,\mcl{K}_2}: T \pi_1(a) = \pi_2(a)T\text{ for all }a\in\mcl{A}\}.
    \end{align*}
    If there exists a unitary operator in $\mcl{I}(\pi_1,\pi_2)$, then we say that $\pi_1$ and $\pi_2$ are \textit{unitarily equivalent}. Furthermore, we say that $\pi_1$ and $\pi_2$ are \textit{disjoint} if no non-zero sub-representation of $\pi_1$ is unitarily equivalent to any sub-representation of $\pi_2$.

\begin{proposition}[{\cite[Proposition  2.1.4]{Arv76}}]\label{prop-disjoint-iff} 
    Let $\mcl{K}_1,\mcl{K}_2$ be real Hilbert spaces and $\pi_j:\mcl{A}\to\B{\mcl{K}_j}, j=1,2$, be two representations. Then $\pi_1$ and $\pi_2$ are disjoint if and only if $\mcl{I}(\pi_1,\pi_2)=\{0\}$. 
\end{proposition}

    

    Note that if $\pi_1$ and $\pi_2$ are two disjoint representations, then from the above proposition, it follows that $\oplus_1^m\pi_1$ and $\oplus_1^n\pi_2$ are disjoint representations for all $m,n\in\mbb{N}$. Also note that the complexifications $(\pi_1)_c$ and $(\pi_2)_c$ are disjoint.

\begin{definition}
    Let  $\Phi_j:\mcl{A}\to\B{\mcl{H}_j}$ be UCP maps with minimal Stinespring representation $(\mcl{K}_j,\pi_j,V_j)$ for $j=1,2$. We say that $\Phi_1$ and $\Phi_2$ are \textit{disjoint} if $\pi_1$ and $\pi_2$ are disjoint representations. 
    A countable family $\{\Phi_j\}_j$ of UCP maps $\Phi_j:\mcl{A}\to\B{\mcl{H}_j}$, each with minimal Stinespring representation $(\mcl{K}_j,\pi_j,V_j)$, is said to be \textit{mutually disjoint} if $\Phi_i$ and $\Phi_j$ are disjoint for all $i\neq j$. 
\end{definition} 



\begin{proposition}[{\cite[Proposition 3.4]{BhKu22}}]\label{prop-disjoint-$C^*$-ext} 
    Let $\{\Phi_j\}_j$ be a countable family of mutually disjoint UCP maps $\Phi_j:\mcl{A}\to\B{\mcl{H}_j}$, where $\mcl{H}_j$'s are real Hilbert spaces. Then $\bigoplus_j\Phi_j\in \mathrm{UCP}_{C^*-ext}(\mcl{A},\B{\bigoplus_j\mcl{H}_j})$ if and only if each $\Phi_j \in \mathrm{UCP}_{C^*-ext}(\mcl{A},\B{\mcl{H}_j})$.
\end{proposition}

    In \cite[Theorem 2.1]{FaZh98}, the necessary and sufficient conditions are provided for a UCP map on a complex $C^*$-algebra $\mcl{A}$ to be a $C^*$-extreme point of $\mathrm{UCP}(\mcl{A}, \B{\mcl{H}})$, where $\mcl{H}$ is a complex Hilbert space. In the remainder of this section, we examine whether these conditions are also necessary and sufficient in the context of real $C^*$-algebras. We observe that subtle differences arise between the real and complex $C^*$-algebra cases. To this end, we recall the following terminology and some observations from the literature.

\begin{definition}
    Let $\Phi\in\mathrm{UCP}(\mcl{A},\B{\mcl{H}})$. A unital CP map $\Psi:\mcl{A}\to\B{\mcl{G}}$, where $\mcl{G}$ is a real Hilbert space, is said to be a \it{compression} of $\Phi$ if there exists an isometry $W: \mcl{G} \to \mcl{H}$ such that $\Psi= \mathrm{Ad}_W\circ\Phi$. A \it{nested sequence of compressions} of a representation $\pi:\mcl{A}\to\B{\mcl{H}}$  is a (finite) sequence $\{\Phi_j\}_j$ of UCP maps $\Phi_j:\mcl{A}\to\B{\mcl{H}_j}$ such that $\Phi_{j+1}$ is a compression of $\Phi_j$, for each $j\geq 1$, and $\Phi_1$ is a compression of $\pi$.
\end{definition}

    The following observations follow from the definition:  
    \begin{itemize}
     \item Let $\Phi\in\mathrm{UCP}(\mcl{A},\B{\mcl{H}})$ be a pure UCP map with the minimal Stinespring representation $(\mcl{K},\pi,V)$, and let $\Psi=\mathrm{Ad}_{W}\circ\Phi$ be a compression of $\Phi$ for some isometry $W\in\B{\mcl{G,H}}$. Then $(\mcl{K},\pi,VW)$ is the minimal Stinespring representation for $\Psi$, and consequently, $\Psi$ is pure. 
     \item Let $\Phi_j\in\mathrm{UCP}(\mcl{A},\B{\mcl{H}_j}), j=1,2$, and let $(\mcl{K},\pi,V_j)$ be their minimal Stinespring representation (i.e., both $\Phi_1$ and $\Phi_2$ are compression of $\pi$). Then $\Phi_2$ is a compression of $\Phi_1$ if and only if $V_2V_2^* \leq V_1V_1^*$, i.e., $\ran{V_2}\subseteq\ran{V_1}$.
    \end{itemize}


\begin{lemma}
    Let $\Phi_j \in \mathrm{UCP}(\mcl{A},\B{\mcl{H}_j})$, $j=1,2$, be compressions of the same irreducible representation $\pi:\mcl{A}\to\B{\mcl{K}}$ of real type, where $\dim(\mcl{H}_2)\leq\dim(\mcl{H}_1)<\infty$. If both $\Phi_1$ and $\Phi_2$ are pure UCP maps and $\Phi_1 \oplus \Phi_2 \in \mathrm{UCP}_{C^*-ext}(\mcl{A},\B{\mcl{H}_1\oplus\mcl{H}_2})$, then $\Phi_2$ is a compression of $\Phi_1$.
\end{lemma}

\begin{proof}
    Let $(\mcl{K},\pi,V_j)$ be the minimal Stinespring representation for $\Phi_j, j=1,2$. Then $(\wtilde{\mcl{K}},\wtilde{\pi},V)$ is the minimal Stinespring representation for $\Phi_1\oplus\Phi_2$, where $\wtilde{\mcl{K}} = \mcl{K}\oplus\mcl{K}$, $\wtilde{\pi} = \pi \oplus\pi$ and $V=V_1\oplus V_2$. Let $D:= \Matrix{2I_{\mcl{K}}&I_{\mcl{K}}\\ I_{\mcl{K}} & 2I_{\mcl{K}}} \in \wtilde{\pi}(\mcl{A})'$, which is a positive invertible operator. Since $V$ is an isometry, we have $V^*DV = \Matrix{2I_{\mcl{H}_1} & V_1^*V_2\\ V_2^*V_1 & 2I_{\mcl{H}_2}}\in\B{\mcl{H}\oplus\mcl{H}}$ is positive and invertible. As $\Phi_1\oplus\Phi_2$ is a $C^*$-extreme point, by Theorem \ref{thm-Cext-abstract-char}, there exists $S\in\wtilde{\pi}(\mcl{A})'$ such that $S^*S=D$ and $SVV^*=VV^*SVV^*$. Since $\pi$ is of real type, we have $\wtilde{\pi}(\mcl{A})' = (I_2 \otimes \pi(\mcl{A}))' = M_2(\mbb{R}) \otimes I_{\mcl{K}}$. So, there exist scalars $\alpha_j\in\mbb{R}$ such that $S=\Matrix{\alpha_1 I_{\mcl{K}} & \alpha_2 I_{\mcl{K}} \\ \alpha_3 I_{\mcl{K}} & \alpha_4 I_{\mcl{K}}}$. Then
    \begin{align}\label{eq-SVV}
        SVV^*=VV^*SVV^* \implies \Matrix{\alpha_1V_1V_1^* & \alpha_2V_2V_2^* \\ \alpha_3V_1V_1^* & \alpha_4V_2V_2^*} = \Matrix{\alpha_1V_1V_1^*V_1V_1^* & \alpha_2V_1V_1^*V_2V_2^* \\ \alpha_3V_2V_2^*V_1V_1^* & \alpha_4V_2V_2^*V_2V_2^*}.
    \end{align}
    Also, $S^*S= D$ implies that $\alpha_2\neq 0$ or $\alpha_3 \neq 0$. Then, from \eqref{eq-SVV}, it follows that there are two possibilities: 
    \begin{align*}
        V_2V_2^* = V_1V_1^*V_2V_2^*
        \quad\mbox{ or }\quad
        V_1V_1^* = V_2V_2^*V_1V_1^*.
    \end{align*}
    The above identities implies that $\ran{V_2}\subseteq\ran{V_1}$ or $\ran{V_1}\subseteq\ran{V_2}$. Since $\dim(\mcl{H}_j)=\dim(\ran{V_j})$ and $\dim{\mcl{H}_2} \leq \dim{\mcl{H}_1}$, we must have $\ran{V_2}\subseteq\ran{V_1}$. Therefore, $\Phi_2$ is a compression of $\Phi_1$.
\end{proof}

\begin{remark}\label{note-compression}
    If $\mcl{H}$ is finite-dimensional and $\Phi \in \mathrm{UCP}_{C^*-ext}(\mcl{A},\B{\mcl{H}})$, then by Theorem \ref{thm-directsum-pure}, $\Phi$ is unitarily equivalent to a direct sum of pure UCP maps; i.e., there exist finitely many pairwise non-equivalent irreducible representations $\pi_1,\pi_2,\cdots,\pi_k$ of $\mcl{A}$, subspaces $\mcl{H}^i_j$ of $\mcl{H}$ ($1\leq j\leq n_i$), and compressions $\Phi_j^{\pi_i} : \mcl{A}\to\B{\mcl{H}^i_j}$ ($1\leq j\leq n_i$) of each representation $\pi_i$ such that 
    \begin{align}\label{eq-main-H-decomp}
        \mcl{H} = \oplus_{i=1}^k(\oplus_{j=1}^{n_i} \mcl{H}^i_j)
    \end{align}
    and with respect to this decomposition of $\mcl{H}$,
    \begin{align}\label{eq-main-Phi-decomp}
        \Phi \sim \oplus_{i=1}^k (\oplus_{j=1}^{n_i} \Phi_j^{\pi_i}). 
    \end{align}
    Note that since $\pi_i$ is irreducible, each $\Phi_j^{\pi_i}$ is a pure UCP map. Now, for each $1 \leq i \leq k$, we arrange the Hilbert spaces $\mcl{H}^i_1,\cdots,\mcl{H}^i_{n_i}$ such that $\dim(\mcl{H}^i_{j+1})\leq \dim(\mcl{H}^i_j)$ for all $1\leq j\leq n_i-1$. Also note that $\{\oplus_{j=1}^{n_i} \Phi_j^{\pi_i}\}_{i=1}^k$ is a family of mutually disjoint UCP maps as $\{\pi_i\}_{i=1}^k$ is mutually disjoint.
\end{remark}


\begin{theorem}\label{thm-FaZh-MAIN-realtype}
    Let $\mcl{H}$ be finite-dimensional and $\Phi \in \mathrm{UCP}_{C^*-ext}(\mcl{A},\B{\mcl{H}})$. Let $\pi_i$ ($1\leq i \leq k$) , $\mcl{H}_j^i , \Phi_j^{\pi_i}$ ($1\leq j\leq n_i$) be as in \eqref{eq-main-H-decomp} and \eqref{eq-main-Phi-decomp}. If the irreducible representations $\pi_i$'s are of real type, then $\{\Phi_j^{\pi_i}\}_{j=1}^{n_i}$ is a nested sequence of compressions for all $1\leq i\leq k$.
\end{theorem}

\begin{proof}
    Since $\Phi$ is $C^*$-extreme, it follows from Proposition \ref{prop-disjoint-$C^*$-ext} that  the direct sum $\bigoplus_{j=1}^{n_i} \Phi_j^{\pi_i}$ is a $C^*$-extreme point for each $1\leq i \leq k$. 
    Now fix $1\leq i\leq k$. If $n_i =1$, then there is nothing to prove. So, assume that $n_i \geq 2$. Then, for each $1\leq j\leq n_i$, define $\Phi_1 = \Phi_j^{\pi_i} \oplus \Phi_{j+1}^{\pi_i}$ and let $\Phi_2$ be the remaining direct sums on the right-hand side of equation \eqref{eq-main-Phi-decomp}. Then, by Proposition \ref{prop-CompExtremepoint},  we conclude that $\Phi_j^{\pi_i} \oplus \Phi_{j+1}^{\pi_i}$ is a $C^*$-extreme point of $\mathrm{UCP}(\mcl{A},\B{\mcl{H}_j^i \oplus \mcl{H}_{j+1}^i})$. By the above lemma, $\Phi_{j+1}^{\pi_i}$ is a compression of $\Phi_j^{\pi_i}$. This completes the proof.
\end{proof}

    It is unclear whether the assumption that each $\pi_i$ is of real type is necessary for the above theorem. We also do not have a counterexample to demonstrate otherwise. However, in \cite[Theorem 2.1]{FaZh98}, it is shown that the same conclusion holds for complex $C^*$-algebras without requiring any additional assumptions on the irreducible representations $\pi_i$'s.


    Next, we analyze whether the sufficient condition provided in \cite[Theorem 2.1]{FaZh98} for a UCP map on complex $C^*$-algebra to be a $C^*$-extreme point also holds in the real $C^*$-algebra case. 

\begin{theorem}\label{thm-MAIN-FaZh-realtype-converse}
    Let $\pi_1,\pi_2,\dots,\pi_k$ be pairwise non-equivalent irreducible representations of $\mcl{A}$ and $\{ \Phi_j^{\pi_i}:\mcl{A}\to\B{\mcl{H}_j^i}\}_{j=1}^{n_i}$ be nested sequences of compressions of each representation $\pi_i$, where each $\mcl{H}_j^i$ is a finite-dimensional real Hilbert space. If each $\pi_i$ is of real type, then the direct sum 
    \begin{align*}
        \Phi:=\bigoplus_{i=1}^k \Big(\bigoplus_{j=1}^{n_i} \Phi_j^{\pi_i}\Big)
    \end{align*}
    is a $C^*$-extreme point of $\mathrm{UCP}(\mcl{A},\B{\mcl{H}})$, where $\mcl{H}=\bigoplus_{i=1}^k\bigoplus_{j=1}^{n_i}\mcl{H}_j^i$. (Note that $\Phi_j^{\pi_i}$'s are necessarily pure UCP maps.)
\end{theorem}

    
\begin{proof}
    Since $\pi_i$'s are pairwise disjoint irreducible representations of $\mcl{A}$ on real Hilbert spaces, say $\mcl{K}_i$, and are of real type, it follows that their complexifications $(\pi_i)_c$'s are pairwise disjoint irreducible representations of the complex $C^*$-algebra $\mcl{A}_c$. Note that each $\Phi_j^{\pi_i}$ has the minimal Stinespring representation $(\mcl{K}_i,\pi_i,V_j^i)$ for some isometry $V_j^i\in\B{\mcl{H}_j^i,\mcl{K}_i}$. Then, it follows that $\big((\mcl{K}_i)_c,(\pi_i)_c,(V_j^i)_c\big)$ is the minimal Stinespring representation of the complexification $(\Phi_j^{\pi_i})_c$; since $\pi_i$ is of real type, $(\Phi_j^{\pi_i})_c$ is a pure map. Furthermore, since $\{\Phi_j^{\pi_i}\}_{j=1}^{n_i}$ is a nested sequence of compressions of $\pi_i$, we have that $\{(\Phi_j^{\pi_i})_c\}_{j=1}^{n_i}$ is a nested sequence of compressions of $(\pi_i)_c$. Hence, from \cite[Theorem 2.1]{FaZh98}, it follows that 
    \begin{align*}
        \Phi_c=\bigoplus_{i=1}^k \Big(\bigoplus_{j=1}^{n_i} (\Phi_j^{\pi_i})_c\Big)
    \end{align*}
    is a $C^*$-extreme point of $\mathrm{UCP}(\mcl{A},\B{\mcl{H}})$. Finally, from Proposition \ref{prop-Phi_c-Phi-C^*ext}, it follows that $\Phi\in\mathrm{UCP}_{C^*-ext}(\mcl{A},\B{\mcl{H}})$. 
\end{proof}

    The assumption that each $\pi_i$ is of real type cannot be omitted from the above theorem. Specifically, the inflation of the pure real state in Example \ref{eg-inflation-complex} (resp. \ref{eg-inflation-quaternion}) is not a $C^*$-extreme point, where only one $\pi_i$ is involved, and is of complex type (resp. quaternionic type). Here is another example. 
    
\begin{example}
    Consider $\Phi:\mbb{H}\to M_4(\mbb{R})$ defined by 
    \begin{align*}
        \Phi(\alpha+\beta\textbf{i}+\gamma\textbf{j}+\delta\textbf{k}):= \Matrix{\alpha&-\beta&0&0\\\beta&\alpha&0&0\\0&0&\alpha&-\beta\\0&0&\beta&\alpha}.
    \end{align*} 
    Observe that $\Phi=\frac{1}{2}\Phi_1 + \frac{1}{2}\Phi_2$,  where $\Phi_1,\Phi_2:\mbb{H}\to M_4(\mathbb{R})$ are the UCP maps (in fact $\ast$-homomorphisms) given by $\Phi_1=\pi$ as in \eqref{eq-piH}, and 
    \begin{align*}
        \Phi_2(\alpha+\beta\textbf{i}+\gamma\textbf{j}+\delta\textbf{k}):=\Matrix{\alpha&-\beta&\gamma&\delta\\
                   \beta&\alpha&\delta&-\gamma\\
                   -\gamma&-\delta&\alpha&-\beta\\
                   -\delta&\gamma&\beta&\alpha
                  }.
    \end{align*}       
    Thus, $\Phi$ is not a linear extreme and hence not a $C^*$-extreme either. Note that $\pi$ is an irreducible representation of quaternionic type and $\Phi_j^{\pi}:\mbb{H}\to M_2(\mbb{R}), j=1,2$, defined by 
    \begin{align*}
        \Phi_j^{\pi}(\alpha+\beta\textbf{i}+\gamma\textbf{j}+\gamma\textbf{k}):=\Matrix{\alpha&-\beta\\\beta&\alpha}
    \end{align*}
    are compressions of $\pi$ such that $\Phi=\Phi_1^{\pi}\oplus\Phi_2^{\pi}$.
\end{example}
 


\begin{corollary}
    Let $\Phi\in\mathrm{UCP}(M_n(\mbb{R}),M_m(\mbb{R}))$. Then $$\Phi\in\mathrm{UCP}_{C^*-ext}(M_n(\mbb{R}),M_m(\mbb{R}))$$ if and only if there exists a nested sequence of compressions $\{ \Phi_j\}_{j=1}^{k}$ of the identity representation on $M_n(\mbb{R})$ such that $\Phi$ is unitarily equivalent to the direct sum $\bigoplus_{j=1}^k \Phi_j$.
\end{corollary}

\begin{proof}
    This follows from Theorems \ref{thm-FaZh-MAIN-realtype}, \ref{thm-MAIN-FaZh-realtype-converse}, and the fact that any irreducible representation on $M_n(\mbb{R})$ is unitarily equivalent to the identity map on $M_n(\mbb{R})$, which is of real type. 
\end{proof}

\section{Commutative real $C^*$-algebra case}\label{sec-CommrealCalgebra}
 
 Throughout this section, let $\mcl{H}$ be a real Hilbert space and $\mcl{A}$ be a unital commutative real $C^*$-algebra, say $\mcl{A}=C(\Omega,-)$ for some compact Hausdorff topological space $\Omega$ and a homeomorphism $-$ of $\Omega$ of period $2$. 
 Recall that for any finite subsets $\{w_j\}_{j=1}^n\subseteq\Omega$ and $\{\lambda_j\}_{j=1}^n\subseteq\mbb{C}$, Urysohn's lemma guarantees the existence of a function $f \in C(\Omega)$ such that $f(w_j)= \lambda_j$, for $1\leq j \leq n$. 


\begin{lemma}\label{lem-points-separation-1} 
    Let $w\in \Omega$. Then $w=\ol{w}$ if and only if $f(w) \in \mbb{R}$ for all $f \in C(\Omega,-)$. 
\end{lemma}

\begin{proof}
    Suppose that $w\neq \ol{w}$. Choose a function $f \in C(\Omega)$ such that $f(w)=i$ and $f(\ol{w})=0$. Define $g:\Omega\to \mbb{C}$ by $g(\xi) := f(\xi) + \ol{f(\ol{\xi})}$ for all $\xi\in\Omega$. Then $g \in C(\Omega,-)$ and $g(w)=i\notin\mbb{R}$.
\end{proof}

\begin{lemma}\label{lem-points-separation-2} 
    Let $w_1,w_2\in\Omega$ be such that $w_1\notin\{w_2,\ol{w_2}\}$.
    \begin{enumerate}[label=(\roman*)]
        \item If $w_j\neq\ol{w_j},j=1,2$, then for any $\lambda_1,\lambda_2\in\mbb{C}$, there exist $f\in C(\Omega,-)$ such that $f(w_j)=\lambda_j$ for all $j=1,2$.
        \item If $w_1=\ol{w_1}$ and $w_2\neq\ol{w_2}$, then for any $\lambda_1\in\mbb{R}$ and $\lambda_2\in\mbb{C}$, there exists $f\in C(\Omega,-)$ such that $f(w_j)=\lambda_j$ for all $j=1,2$.
        \item If $w_j=\ol{w_j},j=1,2$, then for any $\lambda_1,\lambda_2\in\mbb{R}$, there exists $f\in C(\Omega,-)$ such that $f(w_j)=\lambda_j$ for all $j=1,2$.
    \end{enumerate}   
\end{lemma}

\begin{proof}
    $(i)$ Choose $f\in C(\Omega)$ that satisfies $f(w_j) = \lambda_j$ and $f(\ol{w_j})=0$ for $j=1,2$.  Define $g:\Omega\to \mbb{C}$ by $g(w) := f(w) + \ol{f(\ol{w})}$ for all $w\in\Omega$. Then $g \in C(\Omega,-)$ and $g(w_j)=\lambda_j$. \\
    $(ii)$  Choose $f\in C(\Omega)$ that satisfies $f(w_1)=\frac{\lambda_1}{2}$, $f(w_2)=\lambda_2$ and $f(\ol{w_2})=0$. Define $g:\Omega\to \mbb{C}$ by $g(w) := f(w) + \ol{f(\ol{w})}$ for all $w\in\Omega$. Then $g \in C(\Omega,-)$ and $g(w_j)=\lambda_j$. \\
    $(iii)$  Choose $f\in C(\Omega)$ that satisfies $f(w_j)=\frac{\lambda_j}{2}$ for $j=1,2$. Define $g:\Omega\to \mbb{C}$ by $g(w) := f(w) + \ol{f(\ol{w})}$ for all $w\in\Omega$. Then $g \in C(\Omega,-)$ and $g(w_j)=\lambda_j$.
\end{proof} 

\begin{proposition}\label{prop-ker-of-irr-red}
    Let $\mcl{A}$ be a commutative real $C^*$-algebra and $\pi:\mcl{A} \to \B{\mcl{H}}$ be an irreducible representation. Then $\ker{\pi}$ is a maximal ideal of $\mcl{A}$. In fact, $\mcl{A}/\ker{\pi}\cong \pi(\mcl{A})\cong \mbb{R}$ or $\mbb{C}$.
\end{proposition}

\begin{proof}
    Let $\pi:\mcl{A}\to \B{\mcl{H}}$ be any representation. Then $J := \ker{\pi}$ is a closed two-sided ideal of $\mcl{A}$ and by \cite[Proposition 5.4.1]{Li03}, $\mcl{A}/J$ is a real $C^*$-algebra. Then the map $\pi$ factors through $\mcl{A}/J$ in the natural way:
    \begin{equation*}
        \begin{tikzcd}
            \mcl{A} \arrow{r}{\mbf{q}} \arrow[swap]{dr}{\pi} & \mcl{A}/J \arrow{d}{\wtilde{\pi}} \\ & \B{\mcl{H}}
        \end{tikzcd}
    \end{equation*}
    Here $\mbf{q}$ is the quotient map and $\widetilde{\pi}(f+J):=\pi(f)$ for all $f\in \mcl{A}$. Note that $\widetilde{\pi}$ is a faithful representation of $\mcl{A}/J$ on $\mcl{H}$. Then the representation $\widetilde{\pi}_c$ is faithful, which implies that $\widetilde{\pi}_c$ is isometric. Consequently, $\widetilde{\pi}$ is isometric. Then $\widetilde{\pi}(\mcl{A}/J) = \pi(\mcl{A})$ is a real $C^*$-algebra and hence $\mcl{A}/J \cong \pi(\mcl{A})$. 
    Since $\mcl{A}$ is commutative, $\pi(\mcl{A})$ is commutative. Then we have $\pi(\mcl{A}) \subseteq \pi(\mcl{A})'$. Now assume that $\pi$ is irreducible. From \cite[Proposition 5.3.7]{Li03}, we have $\pi(\mcl{A})'\cap\B{\mcl{H}}_{sa}=\mbb{R}I_\mcl{H}$. Then for any $0 \neq T \in \pi(\mcl{A})\subseteq\pi(\mcl{A})'$, we have $T^*T = TT^* = \norm{T}^2 I$. This shows that $T$ is invertible and $T^{-1} = \frac{1}{\norm{T}^2} T^* \in \pi(\mcl{A})$. Therefore, $\pi(\mcl{A})$ is divisible and hence by \cite[Proposition 5.6.6]{Li03}, $\pi(\mcl{A}) \cong \mbb{R} \text{ or } \mbb{C}$. In particular, $J$ is a maximal ideal of $\mcl{A}$ because $\mcl{A}/J$ is divisible.
\end{proof}

\begin{lemma}\label{lem-irr-rep-dim}
    Let $\pi: \mbb{C} \to \B{\mcl{H}}$ be an irreducible representation. Then $\dim{\mcl{H}} =2$ and $\pi:\mbb{C}\to \B{\mcl{H}}\cong M_2(\mbb{R})$ is given by
    \begin{align*}
        \pi(\lambda)=\Matrix{Re \lambda &&  \pm Im \lambda \\ \mp Im \lambda && Re \lambda},
        \qquad\forall~\lambda\in\mbb{C}.
    \end{align*}
\end{lemma}

\begin{proof}
    Let $\pi:\mbb{C}\to\B{\mcl{H}}$ be any representation. Since $\pi$ is unital and self-adjoint, we have $\pi(1)=I$ and $\pi(i) = S$, where $S\in\B{\mcl{H}}$  with $S^*=-S$. Moreover, since $\pi$ is multiplicative, we must have
    \begin{align*}
        S^*S 
             = \pi(-i\cdot i)
             = I 
             = \pi(i\cdot -i) 
             = SS^*. 
    \end{align*}
    Thus, $S\in\B{\mcl{H}}$ is an anti-symmetric unitary. Since $\pi$ is a real linear map on $\mbb{C}$,
    \begin{align*}
        \pi(\lambda)= Re (\lambda) I + Im (\lambda) S,\qquad\forall~\lambda\in\mbb{C}.
    \end{align*}
    Now, let $0\neq v\in\mcl{H}$. Since $S$ is a unitary, we get $Sv\neq 0$, and since $S$ is anti-symmetric, it follows that $\ip{Sv,v} = 0$. Thus, $\{v,Sv\}$ is an orthogonal subset of $\mcl{H}$. Clearly, $\lspan{\{v,Sv\}}$ is a two-dimensional $\pi(\mbb{C})$-invariant subspace of $\mcl{H}$. Therefore, if $\pi$ is irreducible, then $\dim{\mcl{H}}=2$. Identifying $\B{\mcl{H}}\cong M_2(\mbb{R})$, there are only two anti-symmetric unitary operators in $\B{\mcl{H}}$, namely $S = \sMatrix{0& \pm1 \\ \mp1 & 0}$. Hence, $\pi$ has the desired form. 
\end{proof}

    The following result is of a fundamental nature and may have been previously established in the literature. However, we could not find a direct reference to it. Therefore, we present and prove it here. 
    
\begin{proposition}\label{prop-irr-repre} 
    Let $\pi: \mcl{A} \to \B{\mcl{H}}$ be an irreducible representation. Then, $dim(\mcl{H})\leq 2$ and $\pi$ must be one of the following form:
    \begin{enumerate}[label=(\roman*)]
        \item Real type: If $\dim(\mcl{H})=1$, then there exists $w=\ol{w}\in\Omega$ such that $\pi:\mcl{A}\to\B{\mcl{H}}\cong\mbb{R}$ is given by 
        \begin{align}\label{eq-irr-rep-1dim}
                  \pi(f)=f(w),\qquad\forall~f\in \mcl{A}.
              \end{align}
        \item Complex type: If $\dim(\mcl{H})=2$, then there exists $w\neq\ol{w}\in\Omega$ such that $\pi:\mcl{A}\to\B{\mcl{H}}\cong M_2(\mbb{R})$ is given by 
        \begin{align}\label{eq-irr-rep-2dim}
                  \pi(f):=\Matrix{Re f(w) & \pm Im f(w)\\ \mp Im f(w) & Re f(w)},\qquad\forall~f\in \mcl{A}.
              \end{align}
    \end{enumerate}
\end{proposition}

\begin{proof}
   Since $\pi$ is irreducible, by Proposition \ref{prop-ker-of-irr-red}, $J:=\ker{\pi}$ is a maximal ideal and $\mcl{A}/J \cong \mbb{R} \text{ or } \mbb{C}$. Note that $\pi$ factors through $\mcl{A}/J$ as $\pi=\widetilde{\pi}\circ\mbf{q}$, where $\mbf{q}:\mcl{A}\to\mcl{A}/J$ is the quotient map and $\widetilde{\pi}:\mcl{A}/J\to\B{\mcl{H}}$ is given by $\widetilde{\pi}(f+J):=\pi(f)$ for all $f\in\mcl{A}$. As $\pi$ is irreducible, it follows that $\wtilde{\pi}$ is an irreducible representation. Also, $\mbf{q}$ is a non-zero multiplicative real linear functional on $\mcl{A}$. Hence, the map $\widetilde{\mbf{q}}:\mcl{A}_c\to\mbb{C}$ defined by 
    \begin{align*}
        \widetilde{\mbf{q}}(f+ig):=\mbf{q}(f)+i\mbf{q}(g),\qquad\forall~f,g\in \mcl{A},
    \end{align*}
    is a non-zero multiplicative complex linear functional. Since $\mcl{A}_c=C(\Omega)$, there exists $w\in\Omega$ such that $\widetilde{\mbf{q}}(f)=f(w)$ for all $f\in C(\Omega)$. In particular, $\mbf{q}(f)=f(w)$ for all $f\in\mcl{A}$.\\
    \ul{Case (1):} Suppose $\mcl{A}/J \cong \mbb{R}$. Then, $f(w)=\mbf{q}(f)\in\mbb{R}$ for all $f\in \mcl{A}$ and hence, by Lemma \ref{lem-points-separation-1}, $w\in\Omega$ must be such that $w=\ol{w}$. Since the only irreducible representation of $\mbb{R}$, up to unitary equivalence, is the identity map on $\mbb{R}$, it follows that $\mcl{H}\cong\mbb{R}$ and $\wtilde{\pi}:\mbb{R}\to\mbb{R}$ is the identity map. Thus, in this case, $\pi:\mcl{A}\to\mbb{R}$ and is given by
    \begin{align*}
        \pi(f)=\widetilde{\pi}\circ\mbf{q}(f)=f(w),\qquad\forall~f\in \mcl{A}.
    \end{align*}
    \ul{Case (2):} Suppose $\mcl{A}/J \cong \mbb{C}$. Since $q$ is surjective, by Lemma \ref{lem-points-separation-1}, we get $w\neq\ol{w}$. Furthermore, since $\widetilde{\pi}$ is an irreducible representation of $\mbb{C}=\mcl{A}/J$, it follows from Lemma \ref{lem-irr-rep-dim} that $\dim(\mcl{H})=2$ and the representation $\widetilde{\pi}:\mbb{C}\to \B{\mcl{H}}\cong M_2(\mbb{R})$ is given by
    \begin{align*}
        \widetilde{\pi}(\lambda)=\Matrix{Re \lambda &&  \pm Im \lambda \\ \mp Im \lambda && Re \lambda},
        \qquad\forall~\lambda\in\mbb{C}.
    \end{align*}
    Thus $\pi:\mcl{A}\to\B{\mcl{H}}\cong M_2(\mbb{R})$ is given by
    \begin{align*}
        \pi(f)=\widetilde{\pi}\circ\mbf{q}(f)=\Matrix{Re f(w) &&  \pm Im f(w) \\ \mp Im f(w) && Re f(w)},\qquad\forall~f\in \mcl{A}.
    \end{align*}
    This completes the proof.
\end{proof}

    It is known, from (\cite[Proposition 5.3.2]{Li03}), that pure real states of $C(\Omega,-)$ have the form $f\mapsto Re f(w)$ for some $w\in\Omega$. The following corollary is a direct consequence of the above proposition.

\begin{corollary} \label{cor-puremaps}
    If $\Phi:\mcl{A}\to \B{\mcl{H}}$ is any pure UCP map, then one of the following happens: 
    \begin{enumerate}[label=(\roman*)]
    \item $\dim(\mcl{H})=1$ and there exists $w\in\Omega$ such that $\Phi:\mcl{A}\to\B{\mcl{H}}\cong\mbb{R}$ is given by 
    \begin{align*}
             \Phi(f):=Re f(w),\qquad\forall~f\in \mcl{A}.
    \end{align*}
    \item $\dim(\mcl{H})=2$ and there exists $w\neq \ol{w}\in\Omega$ such that $\Phi:\mcl{A}\to\B{\mcl{H}}\cong M_2(\mbb{R})$ is given by
    \begin{align}\label{eq-pureUCP-2}
            \Phi(f)=\Matrix{Re f(w) & \pm Im f(w)\\ \mp Im f(w) & Re f(w)},\qquad\forall~f\in \mcl{A}.
    \end{align}  
    \end{enumerate}
\end{corollary}

\begin{proof}
    Let $(\mcl{K},\pi,V)$ be the minimal Stinespring representation of $\Phi$. Since $\pi:\mcl{A}\to\B{\mcl{K}}$ is irreducible, by Proposition \ref{prop-irr-repre}, we have $\mcl{K}\cong\mbb{R}$ or $\mbb{R}^2$.\\
    \ul{Case (1):} Suppose $\mcl{K}\cong\mbb{R}$. Then there exists $w=\ol{w}\in\Omega$ such that $\pi(f)=f(w)$ for all $f\in C(\Omega,-)$. Hence, $\Phi(f)=V^*\pi(f)V= f(w)$ for all $f\in \mcl{A}$.  \\
    \ul{Case (2):} Suppose $\mcl{K}\cong\mbb{R}^2$. Then $\pi$ must be of the form \eqref{eq-irr-rep-2dim}  for some $w\neq \ol{w}\in\Omega$. Since $V:\mcl{H}\to\mcl{K}$ is an isometry, we must have $\mcl{H}\cong\mbb{R}$ or $\mbb{R}^2$. If $\mcl{H}\cong\mbb{R}$, then $V$ can be identified with a matrix, say $V=\sMatrix{s\\t}\in M_{2\times 1}(\mbb{R})$, where $s,t\in\mbb{R}$ such that $s^2+t^2=1$. It can be easily verified that $\Phi(f)=V^*\pi(f)V=Re f(w)$, for all $f\in\mcl{A}$. Now, if $\mcl{H}\cong\mbb{R}^2$, then $V$ can be identified with a unitary matrix in $M_2(\mbb{R})$, i.e., either $V=\sMatrix{\cos(\theta)&\sin(\theta)\\ \sin(\theta)&-\cos(\theta)}$ or $V=\sMatrix{\cos(\theta)&\sin(\theta)\\ -\sin(\theta)&\cos(\theta)}$ for some $\theta\in\mbb{R}$.  A direct computation will show that $\Phi$ has the form \eqref{eq-pureUCP-2}.
\end{proof}

\noindent\textbf{Notation.}
    Let $w\in\Omega$. Define the pure UCP map $\rho_w:\mcl{A}\to\mbb{R}$ by
        \begin{align}\label{eq-pure-UCP-1}
            \rho_w(f):=Re f(w),\qquad\forall~f\in \mcl{A}.
        \end{align}
        If $w\neq\ol{w}$ define the pure UCP map $\Pi_w:\mcl{A}\to M_2(\mbb{R})$ by
        \begin{align}\label{eq-pure-UCP-2}
            \Pi_w(f):=\Matrix{Re f(w) & Im f(w)\\ - Im f(w) & Re f(w)},\qquad\forall~f\in \mcl{A}.
    \end{align}
    Note that for any $w\neq\ol{w}\in\Omega$, we have $\Pi_{\ol{w}}(f) = \sMatrix{Re f(w)& -Im f(w)\\ Im f(w) & Re f(w)}$. Thus, $\Pi_w$ is unitarily equivalent to $\Pi_{\ol{w}}$. Clearly, $\rho_w$ is a compression of $\Pi_w$ for any $w\neq\ol{w}$. 
     
\begin{definition}
    Let $\{\Phi_j\}_j$ be a family of finitely many pure UCP maps $\Phi_j:\mcl{A}\to\B{\mcl{H}_j}$.  For any $w\in\Omega$, the direct sum $\bigoplus_{j}\Phi_j \in \mathrm{UCP}(\mcl{A},\B{\bigoplus_j\mcl{H}_j})$ is said to be of \it{type-$w$} if one of the following  conditions holds:
    \begin{enumerate}[label=(\roman*)]
        \item $w=\ol{w}$ and $\Phi_j = \rho_w$ for all $j$, or
        \item $w\neq\ol{w}$ and $\Phi_j \in \{\rho_w,\Pi_w,\Pi_{\ol{w}} \}$ for all $j$.
    \end{enumerate}
\end{definition}

    It is important to note that, in the above definition, each $\mcl{H}_j$ is necessarily isomorphic to $\mbb{R}$ or $\mbb{R}^2$. Suppose $w=\ol{w}\in\Omega$. Then, by Lemma \ref{lem-points-separation-1}, we have $\rho_w(f) = f(w)$ for all $f\in\mcl{A}$. Consequently, any type-$w$ map is multiplicative, and therefore, a $C^*$-extreme point of $\mathrm{UCP}(\mcl{A},\B{\mcl{H}})$. The following result is the crucial step in characterizing the $C^*$-extreme points of $\mathrm{UCP}(\mcl{A},\B{\mcl{H}})$.  

\begin{proposition} \label{prop-rho-pi-Cextt}
    Let $w\neq \ol{w}\in \Omega$ and $\Phi=\bigoplus_j\Phi_j$ be a type-$w$ UCP map, where $\Phi_j:\mcl{A}\to\B{\mcl{H}_j}$ is pure. Then $\Phi\in \mathrm{UCP}_{C^*-ext}(\mcl{A},\B{\bigoplus_j\mcl{H}_j})$ if and only if $\Phi_j=\rho_w$ for at most one $j$. 
\end{proposition}

\begin{proof} 
    Let $\mcl{H}:=\bigoplus_j\mcl{H}_j$ and $n:=\dim(\mcl{H})\in\mbb{N}$.\\
    $(\Longrightarrow)$ If $n=1$, then there is only one $\Phi_j$, namely $\Phi_j=\rho_w$. If $n=2$, then  $\bigoplus_j\Phi_j= \rho_w\oplus\rho_w$ or $\Pi_w$ or $\Pi_{\ol{w}}$. Clearly, $\rho_w \oplus \rho_w$ is not a $C^*$-extreme point because we can write $\rho_w \oplus \rho_w = \frac{1}{2} \Pi_w + \frac{1}{2}\Pi_{\ol{w}}$. Now,  assume $n>2$. Suppose $\Phi_j = \rho_w$ for at least two indices $j$. Then $\Phi =\Psi \oplus \rho_w \oplus \rho_w$, where $\Psi$ is a direct sum of maps from the set $\{ \rho_w,\Pi_w,\Pi_{\ol{w}} \}$, i.e., of type-$w$. Let $\wtilde{\Phi}_1:= \Psi \oplus \Pi_w$ and $\wtilde{\Phi}_2: = \Psi \oplus \Pi_{\ol{w}}$. Then
    \begin{align*}
        \Phi=\Psi \oplus \rho_w \oplus \rho_w = \frac{1}{2} \wtilde{\Phi}_1 + \frac{1}{2} \wtilde{\Phi}_2.
    \end{align*}
    Since $\Pi_w \nsim \rho_w \oplus \rho_w$, we conclude that $\Phi \nsim \wtilde{\Phi_1}$, which contradicts the fact that $\Phi$ is a $C^*$-extreme point. Hence $\Phi_j=\rho_w$ for at most one $j$.\\
    $(\Longleftarrow)$ Assume that $\Phi_j \neq \rho_w$ for all $j$. Then $\Phi$ is a $\ast$-homomorphism and hence it is $C^*$-extreme. So we assume that $\Phi_j=\rho_w$ for exactly one $j$, and hence by Corollary \ref{cor-puremaps} we conclude that $n\in\mbb{N}$ is odd. Since $\Pi_w \sim \Pi_{\ol{w}}$, without loss of generality, we assume that 
    \begin{align*}
        \Phi\sim\underbrace{\Pi_w \oplus \Pi_w \oplus \cdots \oplus \Pi_w}_{\frac{n-1}{2}\text{-times}}\oplus\rho_w
             = (I_{\frac{n-1}{2}} \otimes \Pi_w) \oplus \rho_w 
             = \Matrix{I_{\frac{n-1}{2}} \otimes \Pi_w & 0_{n-1 \times 1} \\ 0_{1 \times n-1} & \rho_w}.
    \end{align*}
    Let  $V_0:=\sMatrix{1\\0}\in M_{2\times 1}(\mbb{R})$ and define
    \begin{align*}
        V:&=\Matrix{I_{n-1} & 0_{n-1 \times 1} \\ 0_{2 \times n-1} &  V_0}\in M_{(n+1)\times n}(\mbb{R})
        \qquad\text{and}\qquad \\
        \pi:&=\underbrace{\Pi_w \oplus\Pi_w \oplus \cdots \oplus \Pi_w}_{\frac{n+1}{2}\text{-times}} = I_{\frac{n+1}{2}} \otimes \Pi_w = \Matrix{I_{\frac{n-1}{2}} \otimes \Pi_w & 0_{n-1 \times 2} \\ 0_{2\times n-1} & \Pi_w}.
    \end{align*}
    Note that $\Phi(\cdot)\sim V^*\pi(\cdot)V$ so that $(\mcl{K},\pi,V)$ is the minimal Stinespring dilation of $\Phi$, where $\mcl{K}:=\bigoplus_1^{\frac{n+1}{2}}\mbb{R}^2$. Now, let $\Psi \leq_{cp} \Phi$ with $\Psi(1)$ invertible. Then, by Theorem \ref{thm-radon-nikodym}, there exists a positive contraction $D \in \pi(\mcl{A})'$ such that $\Psi(\cdot) = V^*D\pi(\cdot) V$. Write $D=\Matrix{Y&W\\W^*&Z}$ where $Y\in M_{n-1}(\mbb{R}), W\in M_{(n-1)\times 2}(\mbb{R}), Z\in M_2(\mbb{R})$. Since $D$ is a positive contraction, we must have $Y$ and $Z$ are positive contractions. Also, $D\in\pi(\mcl{A})'$ implies that
    \begin{equation}
         \left.\begin{array}{l}
      Y\in \left(I_{\frac{n-1}{2}} \otimes \Pi_w(\mcl{A})\right)'\\
      Z \in \Pi_w(\mcl{A})'\\
      (I_{\frac{n-1}{2}} \otimes \Pi_w(\cdot)) W = W \Pi_w(\cdot)
    \end{array}\right\} \tag{$\ast$}
    \end{equation}
    Since $w \neq \ol{w}$, by Lemma \ref{lem-points-separation-1}, we get 
    $$\Pi_w(\mcl{A}) = \left\{\sMatrix{s&t\\-t&s} : s,t \in \mbb{R}\right\}.$$ 
    Then $\Pi_w(\mcl{A})' = \left\{\sMatrix{s&t\\-t&s} : s,t \in \mbb{R}\right\}$ and we have  $Z= \alpha I_2$ for some $ \alpha \in [0,1]$. By direct computation, we can see that \begin{align*}
     \Psi(\cdot) = V^*D\pi(\cdot)V = \Matrix{Y(I_{\frac{n-1}{2}} \otimes \Pi_w(\cdot)) & W\Pi_w(\cdot)V_0 \\ V_0^*W^*(I_{\frac{n-1}{2}} \otimes \Pi_w(\cdot)) & \alpha \rho_w(\cdot)}.
     \end{align*}
     Now, $V^*DV = \Matrix{Y & WV_0\\V_0^*W^*&\alpha}$ is positive and invertible implies that $Y$ is invertible and the matrix \begin{align*}
        \Matrix{I & -Y^{-1}WV_0 \\ 0 & 1}^* \Matrix{Y & WV_0 \\ V_0^*W^* & \alpha} \Matrix{I & -Y^{-1}WV_0 \\ 0 & 1} 
          = \Matrix{Y & 0 \\ 0 & \alpha - V_0^*W^*Y^{-1}WV_0}
    \end{align*}
    is positive and invertible. 
    In particular, $\beta:= \sqrt{\alpha - V_0^*W^*Y^{-1}WV_0}>0$. Note that $X:=\Matrix{Y^{\frac{1}{2}}&Y^{-\frac{1}{2}}WV_0\\0&\beta}$ is invertible and 
    \begin{align*}
        X^*\Phi(\cdot)X
        &= \Matrix{Y^{\frac{1}{2}}(I_{\frac{n-1}{2}} \otimes \Pi_w(\cdot))Y^{\frac{1}{2}} & Y^{\frac{1}{2}}(I_{\frac{n-1}{2}} \otimes \Pi_w(\cdot))Y^{-\frac{1}{2}}WV_0 \\ V_0^*W^*Y^{-\frac{1}{2}}(I_{\frac{n-1}{2}} \otimes \Pi_w(\cdot))Y^{\frac{1}{2}} & R^*(I_{\frac{n-1}{2}} \otimes \Pi_w(\cdot))R + \beta^2 \rho_w(\cdot)},
    \end{align*}
    where 
    $$R:=Y^{-\frac{1}{2}}WV_0\overset{(say)}{=}\Matrix{r_1&r_2&\cdots&r_{n-1}}^*\in M_{(n-1)\times 1}(\mbb{R}).$$
    We observe that 
    \begin{align*}
        R^*(I_{\frac{n-1}{2}} \otimes \Pi_w(f))R
        &=(r_1^2+\cdots+r_{n-1}^2)~Re f(w) \\
        &=R^*R \rho_w(f) \\
        &= (\alpha-\beta^2)\rho_w(f),\qquad\forall~f\in\mcl{A}.
    \end{align*}
    Thus, from $(\ast)$, we get 
    \begin{align*}
    X^*\Phi(\cdot)X
        = \Matrix{Y(I_{\frac{n-1}{2}} \otimes \Pi_w(\cdot)) & W\Pi_w(\cdot)V_0 \\ V_0^*W^*(I_{\frac{n-1}{2}} \otimes \Pi_w(\cdot)) & \alpha \rho_w(\cdot)} 
        = \Psi(\cdot).
    \end{align*}
    Therefore, by Theorem \ref{thm-Cext-abstract-char}, we conclude that $\Phi$ is a $C^*$-extreme point. 
\end{proof}


\begin{notation}
    Given $n\in\mbb{N}$ and $w\neq \ol{w}$ in $\Omega$, we let $\mcl{E}_w^n$ denote the set of all $C^*$-extreme points of $\mathrm{UCP}(\mcl{A},M_n(\mbb{R}))$ of type-$w$ as given in Proposition \ref{prop-rho-pi-Cextt}.
\end{notation}

\begin{lemma} \label{lem-disjoint}
    Let $w_1,w_2\in\Omega$ be such that $w_1 \notin \{w_2,\ol{w_2}\}$. If $\Phi_j$ is a UCP map of type-$w_j$ for $j=1,2$, then $\Phi_1$ and $\Phi_2$ are disjoint.
\end{lemma}

\begin{proof} 
    Let $(\mcl{K}_j,\pi_j,V_j)$ be the minimal Stinespring representation of $\Phi_j:\mcl{A}\to M_{n_j}(\mbb{R})$ for $j=1,2$. To show that $\mcl{I}(\pi_1,\pi_2)=\{0\}$. However, as seen in the above proof, note that each $\pi_j$ is unitarily equivalent to either $I_{n_j}\otimes\rho_{w_j}$ or $I_{k_j}\otimes\Pi_{w_j}$ depending on whether $w_j=\ol{w_j}$ or $w_j\neq\ol{w_j}$, for some $k_j\leq n_j$. Therefore, it suffices to prove this result in the case when $\pi_j\in\{\rho_{w_j},\Pi_{w_j}\}$ for $j=1,2$.\\
    \ul{Case (1):} Suppose $w_j\neq \ol{w_j}$ for $j=1,2$. In this case $\pi_j=\Pi_{w_j}$ for $j=1,2$. Then
    \begin{align*}
        T\in\mcl{I}(\pi_1,\pi_2)
            &\Longrightarrow T\Pi_{w_1}(f) = \Pi_{w_2}(f)T,\qquad\forall~f\in C(\Omega,-)\\
            &\Longrightarrow T \Matrix{\alpha&\beta\\-\beta&\alpha} = \Matrix{\gamma&\delta\\-\delta&\gamma}T,\qquad\forall~\alpha,\beta,\gamma,\delta\in\mbb{R}\qquad(\mbox{by Lemma }\ref{lem-points-separation-2})\\
            &\Longrightarrow T=0\in M_2(\mbb{R}).
    \end{align*}
   \ul{Case (2):} Suppose $w_1=\ol{w_1}$ and $w_2\neq\ol{w_2}$. In this case $\pi_1=\rho_{w_1}$ and $\pi_2=\Pi_{w_2}$. If $T\in\mcl{I}(\pi_1,\pi_2)$, then by Lemma \ref{lem-points-separation-2}, we have $T\alpha = \Matrix{\beta&\gamma\\-\gamma&\beta}T$ for all $\alpha,\beta,\gamma\in\mbb{R}$. This implies $T=0\in M_{2\times 1}(\mbb{R})$. \\
    \ul{Case (3):}  Suppose $w_j= \ol{w_j}$ for $j=1,2$. In this case $\pi_j=\rho_{w_j}$ for $j=1,2$. Then by Lemma \ref{lem-points-separation-2}, we conclude that $\mcl{I}(\pi_1,\pi_2)=\{0\}\subseteq M_{1\times 1}(\mbb{R})$.
\end{proof}

\begin{theorem}\label{thm-commutative-Cext-structure-MAIN}
    Let $n\in\mbb{N}$ and $\Phi \in \mathrm{UCP}(\mcl{A},M_n(\mbb{R}))$. Then $\Phi \in \mathrm{UCP}_{C^*-ext}(\mcl{A},M_n(\mbb{R}))$ if and only if there exist $p,q \in \mbb{N} \cup \{0\}$, and a subset $\{\xi_j\}_{j=1}^p\cup\{\eta_j\}_{j=1}^q\subseteq\Omega$ of distinct points with $\xi_j=\ol{\xi_j}$ for all $j\leq p$, and $\eta_i\notin\{\ol{\eta_i},\ol{\eta_j}\}$ for all $i, j\leq q$ with $i\neq j$ such that
    \begin{align}\label{eq-directsum-maintheorem} 
        \Phi \sim \left(\bigoplus_{j=1}^p \Phi_{\xi_j}\right) \bigoplus \left(\bigoplus_{j=1}^q \Phi_{\eta_j}\right),
    \end{align} 
    where $\Phi_{\xi_j}=\bigoplus_1^{n_j}\rho_{\xi_j}$ for all $j\leq p$,  $\Phi_{\eta_j}\in\mcl{E}_{\eta_j}^{m_j}$ for all $j\leq q$, and $n_j,m_j\in\mbb{N}$ are such that $\sum_{j=1}^pn_j+\sum_{j=1}^qm_j=n$.
\end{theorem}

\begin{proof}
    Let $\Phi \in \mathrm{UCP}_{C^*-ext}(\mcl{A},M_n(\mbb{R}))$. Then, by Theorem \ref{thm-directsum-pure}, $\Phi$ is unitary equivalent to a direct sum of pure UCP maps. From Corollary \ref{cor-puremaps}, we know that pure UCP maps on $\mcl{A}$ are of the form $\rho_w$ for some $w\in\Omega$, or $\Pi_w$ for some $w\neq\ol{w}\in\Omega$. Recall that $\Pi_w\sim\Pi_{\ol{w}}$ for all $w\neq\ol{w}$. Hence there exist $p,q\in\mbb{N}\cup\{0\}$ and a subset of distinct points $\{\xi_j\}_{j=1}^p\cup\{\eta_j\}_{j=1}^q\subseteq\Omega$ with $\xi_i=\ol{\xi_i}$ for all $i\leq p$ and $\eta_j\notin\{\ol{\eta_j},\ol{\eta_k}\}$ for all $j, k\leq q$ with $j\neq k$, such that
    \begin{align*}
        \Phi\sim (\bigoplus_{j=1}^p\Phi_{\xi_j})\bigoplus(\bigoplus_{j=1}^q\Phi_{\eta_j})
    \end{align*}
    where $\Phi_{\xi_j}=\bigoplus_1^{n_j}\rho_{\xi_j}$ is a type-$\xi_j$ map, and 
    \begin{align*}
        \Phi_{\eta_j}=\underbrace{\rho_{\eta_j}\oplus\rho_{\eta_j}\oplus\cdots \rho_{\eta_j}}_{(k_j)-times}\oplus\underbrace{\Pi_{\eta_j}\oplus\Pi_{\eta_j}\oplus\cdots \Pi_{\eta_j}}_{(l_j)-times}
    \end{align*}
    is a type-$\eta_j$ map. By construction and by Lemma \ref{lem-disjoint} we observe that $\{\Phi_{\xi_j}\}_{j=1}^p \cup \{\Phi_{\eta_j}\}_{j=1}^q$ are mutually disjoint UCP maps. Since $\Phi\in\mathrm{UCP}_{C^*-ext}(\mcl{A},M_{n}(\mbb{R}))$, by Proposition \ref{prop-disjoint-$C^*$-ext}, we have $\Phi_{\eta_j} \in \mathrm{UCP}_{C^*-ext}(\mcl{A},M_{k_j+2l_j}(\mbb{R}))$.  Then, by Lemma \ref{prop-rho-pi-Cextt}, we have $k_j\in\{0,1\}$ and $\Phi_{\eta_j} \in\mcl{E}_{\eta_j}^{m_j}$, where $m_j=k_j+2l_j$ for all $1\leq j \leq q$. 
    Conversely, assume that $\Phi$ has the form \eqref{eq-directsum-maintheorem}. Since each direct summand is a $C^*$-extreme point and are mutually disjoint, by Proposition \ref{prop-disjoint-$C^*$-ext}, we have $\Phi \in \mathrm{UCP}_{C^*-ext}(\mcl{A},M_{n}(\mbb{R}))$.
\end{proof}

\begin{corollary}
    If $\mcl{A}=C(\Omega,\mbb{R})$, then the $C^*$-extreme points of $\mathrm{UCP}(\mcl{A},M_n(\mbb{R}))$ are precisely the $\ast$-homomorphisms.
\end{corollary}

    In \cite[Proposition 27]{LoPa81}, the authors considered the $C^*$-convex set $\{S\in\B{\mcl{H}}: S^*=S, \norm{S}\leq 1\}$, where $\mcl{H}$ is a complex Hilbert space. They showed that the $C^*$-extreme points (see \cite{LoPa81} for definition) of this set take the form $2P-I_{\mcl{H}}$, where $P\in\B{\mcl{H}}$ is an orthogonal projection. We observe that their proof works for real Hilbert spaces as well. In particular, if $\mcl{H}=\mbb{R}^n$, then the $C^*$-extreme points correspond to matrices that are unitarily equivalent to the diagonal matrix $\sMatrix{I_k&0\\0&-I_m}\in M_n(\mbb{R})$, for some $k,m\in\mbb{N}\cup\{0\}$ such that $k+m=n$. Here, we consider the set $\mcl{CS}_n(\mbb{R})$ of all contractive skew-symmetric $n\times n$ real matrices, i.e.,
    \begin{align*}
        \mcl{CS}_n(\mbb{R}):= \{ S \in M_n(\mbb{R}): S^* = -S, \norm{S}\leq 1\},
    \end{align*}
    which is a real $C^*$-convex subset of $M_n(\mbb{R})$. We investigate the $C^*$-extreme points of the above set by considering UCP maps from $\mbb{C}$ to $M_n(\mbb{R})$.
    
\begin{remark}
    Given a skew-symmetric matrix $S \in M_n(\mbb{R})$ define  $\Phi_S : \mbb{C} \to M_n(\mbb{R})$ by
    \begin{align*}
        \Phi_S(\lambda) := Re(\lambda)I_n + Im(\lambda)S,\qquad\forall~\lambda\in\mbb{C}.
    \end{align*}
    Note that $\Phi_S$ is a unital self-adjoint positive map; in fact, these are the only unital self-adjoint positive maps from $\mbb{C}$ to $M_n(\mbb{R})$.  Now, consider the complexification $(\Phi_S)_c : \mbb{C}^2 \to M_n(\mbb{C})$ of $\Phi_S$ given by
    \begin{align*}
        (\Phi_S)_c(\lambda,\mu) =\Phi_S\left(\frac{\lambda+\ol{\mu}}{2}\right) + i \Phi_S\left(\frac{\lambda-\ol{\mu}}{2i}\right) =\frac{\lambda}{2} (I_n - iS) + \frac{\mu}{2} (I_n+iS),\qquad\forall~\lambda,\mu\in\mbb{C}.
    \end{align*}
    Since $\mbb{C}^2$ is a commutative complex $C^*$-algebra we have 
       \begin{align*}
            \Phi_S \text{ is CP} 
                \Longleftrightarrow (\Phi_S)_c \text{ is CP}
                \Longleftrightarrow (\Phi_S)_c \text{ is positive} 
                \Longleftrightarrow I_n \pm iS \geq 0 
                \Longleftrightarrow \msf{r}(S)=\norm{S}\leq 1, 
    \end{align*} 
    where $\msf{r}(S)$ is the spectral radius of $S$. Thus, $\Phi_S \in \mathrm{UCP}(\mbb{C}, M_n(\mbb{R}))$ if and only if $\norm{S}\leq 1$. (This also shows that, unlike in the complex case, positive maps on real commutative $C^*$-algebras are not necessarily CP maps.)
\end{remark}    

    
\begin{proposition}
    Let $S\in \mcl{CS}_n(\mbb{R})$. Then the following hold: 
    \begin{enumerate}[label=(\roman*)]
        \item $S \in C^*\text{-}ext(\mcl{CS}_n(\mbb{R}))$ if and only if $\Phi_S \in \mathrm{UCP}_{C^*-ext}(\mbb{C}, M_n(\mbb{R}))$;
        \item $S \in ext(\mcl{CS}_n(\mbb{R}))$ if and only if $\Phi_S \in \mathrm{UCP}_{ext}(\mbb{C}, M_n(\mbb{R}))$.
    \end{enumerate}
    Moreover, in both cases, $\norm{S}=1$. 
\end{proposition}

\begin{proof}
    We will prove $(i)$, and $(ii)$ can be shown in a similar manner.  Assume that $ \Phi_S \in \mathrm{UCP}_{C^*-ext}(\mbb{C}, M_n(\mbb{R}))$. Let $S= \sum_{j=1}^k T_j^*S_jT_j$ be a proper $C^*$-convex decomposition of $S$ with $S_j \in \mcl{CS}_n(\mbb{R})$ and $T_j\in M_n(\mbb{R})$. Then $\Phi_j:=\Phi_{S_j}\in \mathrm{UCP}(\mbb{C}, M_n(\mbb{R})), j=1,2$ are such that $\Phi_S = \sum_{j=1}^n \mathrm{Ad}_{T_j}\circ\Phi_j$ is a proper $C^*$-convex decomposition of $\Phi_S$. Hence, there exist unitaries $U_j \in M_n(\mbb{R})$ such that $\Phi_j = \mathrm{Ad}_{U_j}\circ\Phi_S$ for all $1\leq j\leq n$. This implies that $S_j = U_j^*SU_j$ so that $S$ is a $C^*$-extreme point of $\mcl{CS}_n(\mbb{R})$. Conversely, assume that $S \in C^*\text{-}ext(\mcl{CS}_n(\mbb{R}))$ and let $\Phi_S = \sum_{j=1}^n \mathrm{Ad}_{T_j}\circ\Phi_j$ be a proper $C^*$-convex decomposition of $\Phi_S$ with $\Phi_j \in \mathrm{UCP}(\mbb{C}, M_n(\mbb{R}))$. Note that each $\Phi_j=\Phi_{S_j}$ for some $S_j\in\mcl{CS}_n(\mbb{R})$. Then $S= \sum_{j=1}^k T_j^*S_jT_j$ is a proper $C^*$-convex decomposition of $S$ and, therefore, there exist unitaries $U_j \in M_n(\mbb{R})$ such that $S_j = U_j^*SU_j$. This gives $\Phi_j = \mathrm{Ad}_{U_j}\circ\Phi_S$, concluding that $\Phi$ is a $C^*$-extreme point. 
    
    Now, if $0 < \norm{S} < 1$, then we choose $s,t \in (0,1)\setminus \{\norm{S}\}$ such that $\norm{S}= \frac{1}{2}s + \frac{1}{2}t$. Then $S = \frac{1}{2}(\frac{s}{\norm{S}}S) + \frac{1}{2}(\frac{t}{\norm{S}}S)$, and hence $S$ is neither a $C^*$-extreme point nor an extreme point of $\mcl{CS}_n(\mbb{R})$. Additionally, zero is clearly not an extreme point since $0 = \frac{1}{2}S + \frac{1}{2}S^*$ for any $S\in \mcl{CS}_n(\mbb{R})$. Therefore, $\norm{S}$ must be $1$.
\end{proof}



\begin{corollary}\label{cor-Cstar-ext-CSn}
\mbox{}
    \begin{enumerate}[label=(\roman*)]
        \item If $n$ is odd, then 
        \begin{align*}
            C^*\text{-}ext(\mcl{CS}_n(\mbb{R}))=\Big\{S\in\mcl{CS}_n(\mbb{R}): S\sim 0\oplus\Matrix{0&I_{\frac{n-1}{2}}\\ -I_{\frac{n-1}{2}}&0}\Big\};
        \end{align*}
        \item If $n$ is even, then 
        \begin{align*}
            C^*\text{-}ext(\mcl{CS}_n(\mbb{R}))=\Big\{S\in\mcl{CS}_n(\mbb{R}): S\sim \Matrix{0&I_{\frac{n}{2}}\\ -I_{\frac{n}{2}}&0}\Big\},
        \end{align*}
        which is precisely the set of all skew-symmetric unitary matrices.
    \end{enumerate}
\end{corollary}
    
\begin{proof}
    $(i)$ Recall that $C(\{1,2\},-)\cong\mbb{C}$ as real $C^*$-algebras via the isomorphism $f\mapsto f(1)$, where $\ol{1}=2$ and $\ol{2}=1$. Under this isomorphism, for any  $w\in\{1,2\}$, we have
    \begin{align*}
        \rho_w(\lambda)=Re(\lambda)
        \qquad\mbox{and}\qquad
        \Pi_w(\lambda)\sim\Matrix{Re(\lambda)& Im(\lambda)\\ -Im(\lambda)& Re(\lambda)},
        \qquad\forall~\lambda\in\mbb{C}.
    \end{align*}
    Since $n$ is odd, by Theorem \ref{thm-commutative-Cext-structure-MAIN}, any $\Phi\in\mathrm{UCP}(\mbb{C},M_n(\mbb{R}))$ is a $C^*$-extreme point if and only if 
    \begin{align}\label{eq-Phi-C-Mn(R)}
        \Phi(\lambda)\sim&~ Re(\lambda)\bigoplus\underbrace{\Matrix{Re(\lambda)& Im(\lambda)\\ -Im(\lambda)& Re(\lambda)}\bigoplus\cdots\bigoplus\Matrix{Re(\lambda)& Im(\lambda)\\ -Im(\lambda)& Re(\lambda)}}_{(\frac{n-1}{2})-times},\qquad\forall~\lambda\in\mbb{C} \\
                         &=Re(\lambda) I_n+ Im(\lambda) S \notag,
    \end{align}
    where 
    \begin{align}\label{eq-S-C-Mn(R)}
        S=0\oplus\underbrace{\Matrix{0&1\\-1&0}\oplus\cdots\oplus\Matrix{0&1\\-1&0} }_{(\frac{n-1}{2})-times}
            \quad\sim\quad 0\oplus\Matrix{0&I_{\frac{n-1}{2}}\\ -I_{\frac{n-1}{2}}&0}.  
    \end{align}
    Now, the result follows from the above proposition. \\
    $(ii)$ The proof follows similarly to the case when $n$ is odd. Note that if $n$ is even, then the first term in the direct sum of \eqref{eq-Phi-C-Mn(R)} and \eqref{eq-S-C-Mn(R)} will be absent.
\end{proof}
 
\begin{example}    
    Note that any UCP map from $\mbb{C}$ to $M_2(\mbb{R})$ is of the form $\Phi_S(\lambda) = Re(\lambda)I_2 + Im(\lambda)S$, where $S=\sMatrix{0&t\\-t&0}$ with $\norm{S}=\abs{t} \leq 1$. Consequently, there are only two $C^*$-extreme points and two linear extreme points; they correspond to $\abs{t}=1$, i.e, $t=\pm 1$.
 \end{example}

\begin{example}
    We saw that any UCP map from $\mbb{C}$ to $M_3(\mbb{R})$ is of the form $\Phi_S(\lambda) = Re(\lambda)I_3 + Im(\lambda)S$, for some $S = \sMatrix{0 && \alpha && \beta \\ -\alpha && 0 && \gamma \\-\beta && -\gamma && 0}\in M_3(\mbb{R})$ with $\norm{S}=\sqrt{\alpha^2+\beta^2+\gamma^2} \leq 1$. \\If $\Phi_S \in \mathrm{UCP}_{C^*-ext}(\mbb{C}, M_3(\mbb{R})) $, then from the above proposition we have $\alpha^2 + \beta^2 + \gamma^2 = 1$. However, not all of them correspond to the $C^*$-extreme points. For, by Theorem \ref{thm-commutative-Cext-structure-MAIN}, we know that (up to unitary equivalence) any $C^*-$extreme point is of the form: 
    \begin{align*}
        \lambda \mapsto\Matrix{Re (\lambda) && Im(\lambda) && 0 \\ -Im(\lambda) && Re(\lambda) && 0 \\0 && 0 && Re(\lambda)},
    \end{align*}
    and these maps correspond to $\alpha=1,\beta=0,\gamma=0$. However, it can be seen that the linear extreme points correspond to all $(\alpha,\beta,\gamma) \in \mbb{R}^3$ such that $\alpha^2+\beta^2+\gamma^2=1$. This follows from the fact that the map $S\mapsto (\alpha,\beta,\gamma)$ defines an affine isomorphism from the convex set
    \begin{align*}
        \Big\{S=\sMatrix{0 && \alpha && \beta \\ -\alpha && 0 && \gamma \\-\beta && -\gamma && 0}\in M_3(\mbb{R}): \alpha^2+\beta^2+\gamma^2\leq 1\Big\}
    \end{align*}
    onto the closed unit ball $\{(\alpha,\beta,\gamma) \in \mbb{R}^3: \alpha^2 + \beta^2 + \gamma^2 \leq 1\}$, which has  extreme points precisely $\{(\alpha,\beta,\gamma) \in \mbb{R}^3: \alpha^2 + \beta^2 + \gamma^2 = 1\}$ (c.f. \cite[Theorem B.2]{Lim96}).
\end{example}

\section*{Acknowledgement}
The first author is supported by a CSIR fellowship (File No. 09/084(0780)/2020-EMR-I). The second and third authors acknowledge financial support from the IoE Project of MHRD (India) with reference no SB22231267MAETWO008573. Additionally, the second author is partially supported by the NBHM grant (No. 02011/10/2023 NBHM (R.P) R\&D II/4225).

\bibliographystyle{alpha}

\end{document}